\documentclass[review,onefignum,onetabnum]{siamart171218}

\headers{Method for quasiperiodic elliptic PDEs}
{K. Jiang, M. Li, J. Zhang, and L. Zhang}

\title{Projection method for quasiperiodic elliptic equations and application to quasiperiodic homogenization
\thanks{Submitted to xxx.
}
}
\author{Kai Jiang\thanks{Hunan Key Laboratory for Computation and Simulation in Science and Engineering, Key Laboratory of Intelligent Computing and Information Processing of Ministry of Education, School of Mathematics and Computational Science, Xiangtan University, Xiangtan, Hunan, 411105, China
  (\email{kaijiang@xtu.edu.cn, limeng@smail.xtu.edu.cn, zhangjuan@xtu.edu.cn}).}
\and Meng Li\footnotemark[2]
\and Juan Zhang\footnotemark[2]
\and Lei Zhang\thanks{School of Mathematical Sciences, Institute of Natural Sciences, MOE-LSC, Shanghai Jiao Tong University, Shanghai, 200240, China
  (\email{lzhang2012@sjtu.edu.cn}).}}

\usepackage{mathrsfs,amsmath,amssymb,bm}
\usepackage{lipsum}
\usepackage{amsfonts}
\usepackage{caption,graphicx, subfigure}
\usepackage{epstopdf}
\usepackage{makecell, rotating, bbding}
\usepackage{multirow}
\usepackage{algorithmic, algorithm}
\usepackage{enumerate}
\usepackage{booktabs}
\usepackage[misc]{ifsym}
\usepackage{mathtools}
\usepackage[draft]{changes}
\usepackage{lipsum}
\usepackage{amsopn}

\usepackage{aligned-overset}
\usepackage{geometry}
\usepackage{mathtools}
\usepackage{nicematrix}
\usepackage{tikz}
\usepackage{dashbox}
\usetikzlibrary{calc}
\usepackage{lineno}
\usepackage{hyperref}

\newcommand{\bbR}{\mathbb{R}}

\newtheorem{thm}{Theorem}[section]

\newtheorem{remark}[thm]{Remark}

\newcommand\tbbint{{-\mkern -16mu\int}}

\newcommand\dbbint{{-\mkern -19mu\int}}

\newcommand\bbint{
	{\mathchoice{\dbbint}{\tbbint}{\tbbint}{\tbbint}}
}

\DeclareMathOperator*{\argmin}{\mathrm{argmin}}

\newsiamremark{hypothesis}{Hypothesis}
\crefname{hypothesis}{Hypothesis}{Hypotheses}
\newsiamthm{claim}{Claim}

\makeatletter
\newcommand*{\addFileDependency}[1]{
  \typeout{(#1)}
  \@addtofilelist{#1}
  \IfFileExists{#1}{}{\typeout{No file #1.}}
}
\makeatother

\begin{document}
\maketitle
     
\begin{abstract}
In this study, we address the challenge of solving elliptic equations with quasiperiodic coefficients. To achieve accurate and efficient computation, we introduce the projection method, which enables the embedding of quasiperiodic systems into higher-dimensional periodic systems. To enhance the computational efficiency, we propose a compressed storage strategy for the stiffness matrix by its multi-level block circulant structure, significantly reducing memory requirements. Furthermore, we design a diagonal preconditioner  to efficiently solve the resulting high-dimensional linear system by reducing the condition number of the stiffness matrix. These techniques collectively contribute to the computational effectiveness of our proposed approach. Convergence analysis shows the polynomial accuracy of the proposed method. We demonstrate the effectiveness and accuracy of our approach through a series of numerical examples. Moreover, we apply our method to achieve a highly accurate computation of the homogenized coefficients for a quasiperiodic multiscale elliptic equation. 
\end{abstract}
 
\begin{keywords}
	Quasiperiodic elliptic equation, Projection method, Convergence analysis, Compressed storage, Diagonal preconditioner, Quasiperiodic homogenization.
\end{keywords}
	
\begin{AMS}
	65D05, 65D15, 65F08, 35B15, 35B27
\end{AMS}
	
\section{Introduction}
\label{sec:introduction}

Quasiperiodic systems, such as quasicrystals and Penrose tilings \cite{levine1984quasicrystals, penrose1974role, shechtman1984metallic, senechal1996quasicrystalsn}, share certain similarities with periodic systems like crystals and periodic tilings, particularly in terms of exhibiting long-range order. Nevertheless,  the absence of translational symmetry in quasiperiodic systems, fundamentally distinct from periodic structures, poses challenges in  precisely measuring the effective behavior within a representative volume.


Partial differential equations (PDEs) with quasiperiodic coefficients play a crucial role in describing intriguing phenomena in physics and materials science, such as wave  propagation in quasiperiodic media, Moiré lattices, magic-angle graphene superlattices and irrational interfaces \cite{amenoagbadji2023wave, amenoagbadji2023wave2, blanc2015local, cao2018unconventional, poincare1890problem, shechtman1984metallic, sutton1992irrational}. Research in this area provides essential insights into complex systems. However, the absence of translation invariance and decay in quasiperiodic media necessitates formulating the PDE over the entire space, posing significant numerical challenges. Particularly, in homogenization theory, when microstructures are quasiperiodic, asymptotic homogenization techniques for periodic media \cite{bensoussan2011asymptotic} are extended to quasiperiodic contexts \cite{braides1992homogenization, bouchitte2010homogenization, kozlov1979averaging, nguetseng2003homogenization, wellander2019homogenization, zhikov1987averaging}. This extension requires a limit over the entire space, making direct evaluation difficult.

Addressing these challenges and developing effective numerical methods for PDEs with quasiperiodic coefficients is an active area of research, aiming to enable accurate simulations and computations in quasiperiodic media. Over the years, several numerical approaches have been explored to tackle this problem. One of the most commonly used methods is the periodic approximation method (PAM), which involves approximating quasiperiodic functions with periodic ones,  thereby imposing  periodic boundary conditions for quasiperiodic PDEs.  However, the accuracy of PAM is primarily affected by the Diophantine error \cite{jiang2023approximation}, which measures the distance between a given real number and its approximate rational number. In the quasiperiodic corrector problem, the filtering method (FM) approximates whole-space quasiperiodic elliptic PDEs by restricting them to a finite computational domain, using a  filter function whose derivatives vanish to all orders at the boundary. This construction naturally induces modified homogeneous boundary conditions. From a numerical performance perspective, while this method suppresses oscillations, it still fails to eliminate the influence of  Diophantine errors \cite{blanc2010improving}.




In recent years, the projection method (PM) \cite{jiang2014numerical, jiang2018numerical} has emerged as a highly accurate and efficient approach for approximating quasiperiodic functions. The PM technique captures the essential characteristics of quasiperiodic functions that can be defined on an irrational manifold of a higher-dimensional torus, allowing for their embedding into higher-dimensional periodic functions. Consequently, PM utilizes the spectral collocation method by employing the discrete Fourier-Bohr transformation to approximate quasiperiodic functions. By efficiently computing higher-dimensional periodic parent functions and incorporating a projection matrix, accurate approximations of quasiperiodic functions can be obtained. Extensive research has confirmed the high accuracy of the PM in computing a diverse range of quasiperiodic systems, including quasicrystals \cite{barkan2014controlled, jiang2015stability, si2025designing}, incommensurate quantum systems \cite{gao2023pythagoras, xueyang2021numerical}, topological insulators \cite{wang2022effective}, and grain boundaries \cite{cao2021computing, jiang2022tilt}. Additionally, PM has been proven to possess high accuracy within the quasiperiodic function space and its ability to use the fast Fourier transform (FFT) algorithm to reduce computational complexity \cite{jiang2024numerical}.

In this study, we introduce a highly efficient algorithm for solving the quasiperiodic elliptic equation within the framework of the PM, accompanied by a thorough convergence analysis. The PM  transforms the quasiperiodic elliptic equation from a lower-dimensional problem to a periodic system in a higher dimension, enabling its solution using a pseudo-spectral method.    However, this transformation leads to a large linear system with high-dimensional degrees of freedom, posing computational challenges. To tackle the difficulties associated with the large linear system, we introduce a fast algorithm that utilizes a compressed storage format and a diagonal preconditioner. By employing the compressed storage format, we can effectively store and manipulate the tensor product matrices involved in the linear system. This approach enables efficient handling of the memory-intensive matrices arising from the quasiperiodic problem. Furthermore, the diagonal preconditioner is specifically designed to accelerate matrix-vector multiplications, enhancing the overall efficiency of solving the linear system. Through the combination of the compressed storage format and the diagonal preconditioner, our proposed algorithm significantly improves the computational efficiency of solving the quasiperiodic elliptic equation within the PM framework. 

To assess the effectiveness of our numerical algorithm and verify our theoretical results, we conduct comprehensive tests on various quasiperiodic elliptic equations.  These results demonstrate the remarkable efficiency and polynomial accuracy of our method, emphasizing its ability to avoid the Diophantine approximation error \cite{jiang2023approximation}.  Moreover, we extend the application of our algorithm to compute the homogenized coefficients in a quasiperiodic multiscale problem. By accurately solving the global quasiperiodic corrector equation, our method achieves polynomial accuracy in computing the homogenized coefficients. These outcomes validate the robustness and reliability of our approach in tackling quasiperiodic problems and highlight its potential for a wide range of applications.


\textbf{Organization.} The article is structured as follows. In  \Cref{sec:preliminaries}, we overview the preliminaries of quasiperiodic function spaces and the quasiperiodic elliptic equation. Moreover, we provide a detailed introduction to the PM. In \Cref{sec:numerical_methods}, we apply the PM to discretize quasiperiodic elliptic PDEs and provide a rigorous
 convergence analysis of the proposed method. \Cref{sec:compressed_storage} presents a compressed storage method for the stiffness matrix, which significantly reduces storage requirements by utilizing the multilevel block circulant structure, as well as a diagonal preconditioner for the preconditioned conjugate gradient (PCG) method to accelerate the computation of the discrete system. In  \Cref{sec:numerical_experiments}, we present a series of numerical experiments to validate the accuracy and efficiency of our developed algorithm. These experiments include solving quasiperiodic elliptic equations and computing quasiperiodic homogenized coefficients. Finally, the conclusion and outlook of this paper are given
in \Cref{sec:conclusion}.

\textbf{Notations.} We adopt the following notations throughout the paper. For any two numbers $a$ and $b$, $a\lesssim b$ implies the existence of a constant $C>0$ such that $a\leq Cb$. Similarly, $a \gtrsim b$ can also be defined. 
The Hadamard product $(A\circ B)_{ij}=a_{ij}b_{ij}$ represents the element-wise multiplication, denoted as $\circ$, of two matrices $A=(a_{ij})$, $B=(b_{ij})$ $\in$ $\mathbb{C}^{n\times n}$. 
$\|A\|_F:=\big(\sum_{i=1}^n\sum_{j=1}^n|a_{ij}|^2\big)^{1/2}$ is the Frobenius norm of a matrix $A=(a_{ij})$ $\in$ $\mathbb{C}^{n\times n}$. For a given number $L$, $[L]$ denotes the nearest integer to $L$.  The set $\mathscr{D}(n)$ encompasses  all diagonal matrices of order $n$. $I_n$ represents the identity matrix of order $n$. For a given matrix $\bm{P}=(p_{ij})\in\mathbb{C}^{d\times n}$, its uniform norm $\|\bm{P}\|_{\infty}$ is defined as $\|\bm{P}\|_{\infty}=\max\limits_{1\leq j\leq n}(\sum_{i=1}^d|p_{ij}|)$.  $(\bm{e}_i)_{i=1}^n$ is the canonical basis of $\mathbb{R}^n$.


\section{Preliminaries}\label{sec:preliminaries}

In this section, we present a brief overview of \emph{quasiperiodic functions} and the concept of \emph{quasiperiodic Sobolev spaces} in \Cref{sec:quasiperiodic_elliptic:function}. This preliminary knowledge is essential for establishing the well-posedness of the \emph{quasiperiodic elliptic equation} introduced in \Cref{sec:quasiperiodic_elliptic:equation}. Moveover, we introduce the projection method (PM) in detail in \Cref{sec:numerical_methods:pm}.

\subsection{Quasiperiodic function space}
\label{sec:quasiperiodic_elliptic:function}

We start with an $n$-dimensional periodic function $\mathcal{F}(\bm{y})$, $\bm{y}\in\mathbb{T}^n = (\mathbb{R}/2\pi\mathbb{Z})^n$. The Fourier series of $\mathcal{F}(\bm{y})$ defined on $\mathbb{T}^n$ is 
\begin{equation}\label{eqn:fourier_periodic}
 \mathcal{F}(\bm{y})=\sum_{\bm{k}\in\mathbb{Z}^{n}}\hat{\mathcal{F}}_{\bm{k}}e^{\imath\bm{k}^T\bm{y}},\quad\hat{\mathcal{F}}_{\bm{k}}=\frac{1}{|\mathbb{T}^n|}\int_{\mathbb{T}^n}\mathcal{F}(\bm{y})e^{-\imath\bm{k}^T\bm{y}}d\bm{y},\quad\bm{k}\in\mathbb{Z}^n.
\end{equation}
Next, we introduce the definition and properties of quasiperiodic functions.

\begin{definition}
    A matrix $\bm{P}\in\mathbb{R}^{d\times n} (d\leq n)$ is called the projection matrix, if it belongs to the set $\mathbb{P}^{d\times n }$ defined as $\mathbb{P}^{d\times n}:=\left\{\bm{P}=(\bm{p}_1,\cdots,\bm{p}_n)\in\mathbb{R}^{d\times n}:\bm{p}_1,\cdots,\bm{p}_n ~\mathrm{are}~\mathbb{Q}\text{-linearly independent}\right\}$.
\end{definition}

\begin{definition}
A $d$-dimensional function $f(\bm{x})$ is quasiperiodic, if there exists an $n$ dimensional periodic function $\mathcal{F}$ and a projection matrix $\bm{P}\in\mathbb{P}^{d\times n}$, such that $f(\bm{x})=\mathcal{F}(\bm{P}^T\bm{x})$ for all $\bm{x}\in\mathbb{R}^d$. $\mathcal{F}$ is called the parent function of $f(\bm{x})$.  In particular, when $n = d$ and $\bm{P}$ is nonsingular, $f(\bm{x})$ is periodic. The space of all quasiperiodic functions is denoted as $\mathrm{QP}(\mathbb{R}^d)$.
\end{definition}

As an example, we present an illustration of the  quasiperiodic function $f=\cos(2\pi x)+\cos(2\pi\sqrt{2}x)$, whose projection matrix is $\bm{P} = 2\pi\cdot(1, \sqrt{2})$, as shown in \Cref{fig:1d_quasiperiodic}. 
\begin{figure}[H]
    \centering
        \includegraphics[width=0.75\textwidth]{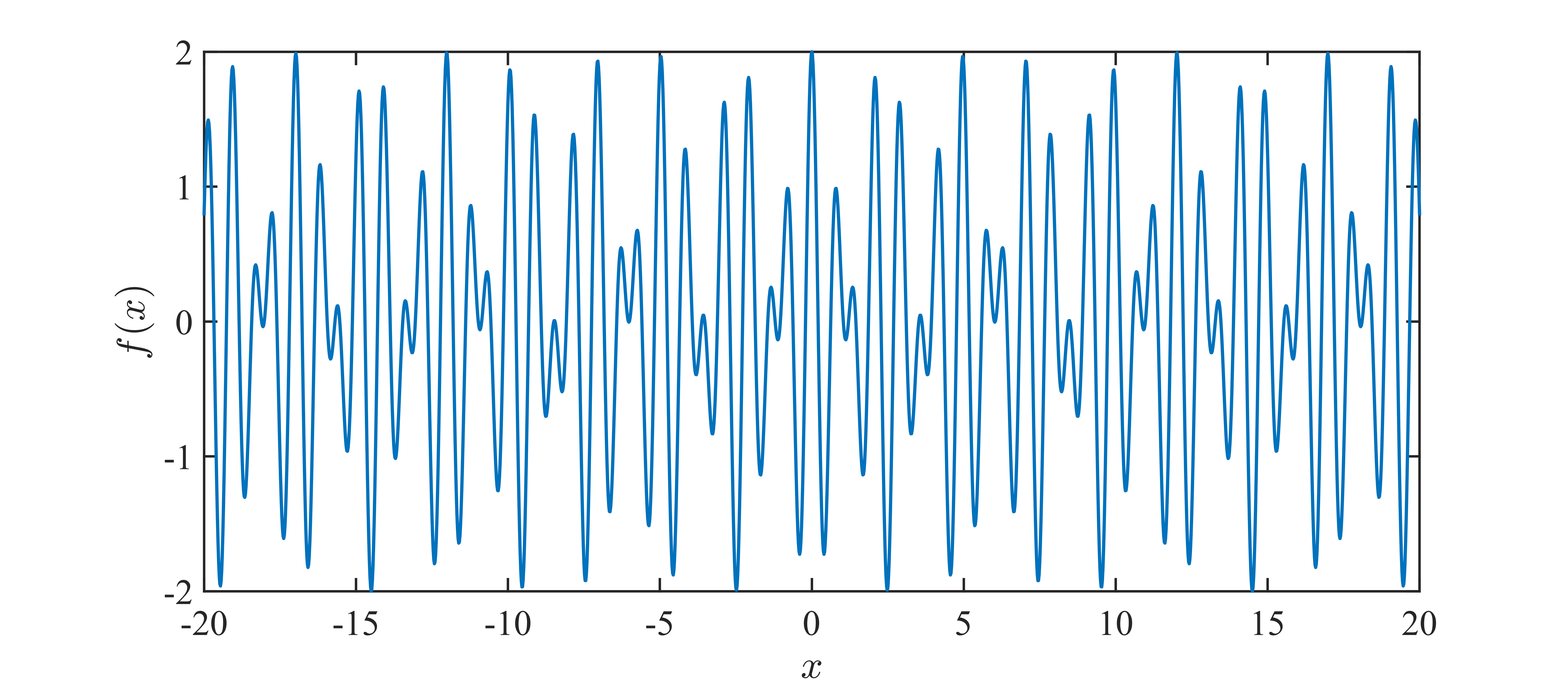}
        \label{fig:image3}
    \caption{{One-dimensional  quasiperiodic function $f=\cos(2\pi x)+\cos(2\pi\sqrt{2}x)$ }}
    \label{fig:1d_quasiperiodic}
\end{figure}

Let $K_T=\left\{\bm{x} = (x_1, \cdots , x_d)\in\mathbb{R}^d : |x_j|\leq T, ~j = 1, \cdots ,d\right\}$. The mean value of $f(\bm{x})\in\mbox{QP}(\mathbb{R}^d)$ is 
\begin{equation}\label{eqn:elliptic:meanvalue}
     \mathcal{M}\left\{ f(\bm{x})\right\} = \lim_{T\to \infty}\frac{1}{(2T)^{d}}\int_{\bm{s}+K_T} f(\bm{x})d\bm{x}:=\bbint f(\bm{x})d\bm{x},
\end{equation}
where the limit on the right side exists uniformly for all $\bm{s}\in\mathbb{R}^d$ \cite{jiang2024numerical}. Therefore, calculating the mean value, by definition, involves a limiting process for an integral defined over infinitely increasing domains, which poses a challenge for numerical algorithms.

Through a simple calculation
\begin{equation}\label{eqn:continue_orthogonality}	\mathcal{M}\Big\{e^{\imath\bm{\lambda}^T_1\bm{x}} e^{-\imath\bm{\lambda}^T_2\bm{x}}\Big\}=
		\left\{
		\begin{aligned}
			&1, ~~~~\bm{\lambda}_1=\bm{\lambda}_2, \\
		  &0,~~~~ \mathrm{otherwise},
		\end{aligned}    			
		\right.
	\end{equation}
and the continuous Fourier transformation 
 \begin{align}
     \hat{f}_{\bm{\lambda_k}}=\mathcal{M}\Big\{ f(\bm{x}) e^{-\imath\bm{\lambda_k}^T\bm{x}}\Big\},
     \label{eqn:continuous_FT}
 \end{align}
the quasiperiodic function $f(\bm{x})$ can be written as the following (generalized) Fourier series, 
\begin{align}\label{eqn:continuous_FS}
f(\bm{x})=\sum_{\bm{k}\in\mathbb{Z}^n}\hat{f}_{\bm{\lambda_k}} e^{\imath\bm{\lambda_k}^T\bm{x}},
\end{align}
where $\bm{\lambda_k}=\bm{P}\bm{k}\in\Lambda$ are Fourier exponents and $\hat{f}_{\bm{\lambda_k}}$ are Fourier coefficients. $\Lambda$ is the spectral point set of $f(\bm{x})$.
 
To clarify the relationship between the Fourier coefficient $\hat{f}_{\bm{\lambda_k}}$ of the low-dimensional quasi-periodic function and the Fourier coefficient $\hat{\mathcal{F}}_{\bm{k}}$ of the higher-dimensional parent function, we present the following \Cref{lemma:consistent_fourier_coefficient}, see also \cite[Theorem 4.1]{jiang2024numerical}, which describes their consistency by invoking the Birkhoff’s ergodic theorem \cite{walters2000introduction}.

\begin{lemma}\label{lemma:consistent_fourier_coefficient}
    For a given quasiperiodic function
$$
f(\bm{x})=\mathcal{F}(\bm{P}^T\bm{x}),~~\bm{x}\in\mathbb{R}^d,
$$
where $\mathcal{F}$ is its parent function defined on $\mathbb{T}^n$ and $\bm{P}$ is the projection matrix, we have
\begin{align}
\hat{f}_{\bm{\lambda_k}}=\hat{\mathcal{F}}_{\bm{k}}.
\end{align}
\end{lemma}

Correspondingly, we introduce the notions of periodic function spaces on $\mathbb{T}^n$ and quasiperiodic function spaces on $\mathbb{R}^d$.  Define the Hilbert space:  $\mathcal{L}^{2}(\mathbb{T}^{n})=\big\{\mathcal{F}(\bm{y}): \displaystyle\frac{1}{\left|\mathbb{T}^{n}\right|} \displaystyle\int_{\mathbb{T}^{n}}|\mathcal{F}|^{2}\, d \bm{y}<\infty\big\}$, equipped with inner product
$
\left(\mathcal{F}_{1}, \mathcal{F}_{2}\right)_{\mathcal{L}^{2}\left(\mathbb{T}^{n}\right)}=\displaystyle\frac{1}{\left|\mathbb{T}^{n}\right|} \displaystyle\int_{\mathbb{T}^{n}}\mathcal{F}_{1} \bar{\mathcal{F}}_{2}\, d \bm{y}.
$
The periodic Sobolev space
 $\mathcal{H}^{s}\left(\mathbb{T}^{n}\right)$ is defined as
$$
\mathcal{H}^{s}\left(\mathbb{T}^{n}\right)=\left\{\mathcal{F} \in \mathcal{L}^{2}\left(\mathbb{T}^{n}\right):\|\mathcal{F}\|_{\mathcal{H}^s(\mathbb{T}^n)}<\infty\right\},
$$
where  $\|\mathcal{F}\|_{\mathcal{H}^s(\mathbb{T}^n)}=\Big(\sum\limits_{\bm{k} \in \mathbb{Z}^{n}}\left(1+|\bm{k}|^{2}\right)^{s}|\hat{\mathcal{F}}_{\bm{k}}|^{2}\Big)^{1 / 2}, \hat{\mathcal{F}}_{\bm{k}}=(\mathcal{F}, e^{\imath \bm{k}^T\bm{y}})_{\mathcal{L}^{2}\left(\mathbb{T}^{n}\right)}$. The semi-norm of  $\mathcal{H}^{s}(\mathbb{T}^{n})$ is given by $|\mathcal{F}|_{\mathcal{H}^s(\mathbb{T}^n)}=\Big(\sum\limits_{\bm{k} \in \mathbb{Z}^{n}}|\bm{k}|^{2 s}|\hat{\mathcal{F}}_{\bm{k}}|^{2}\Big)^{1 / 2}$.

For $s\in\mathbb{N}$, we use $\mathcal{C}^{s}_{QP}(\mathbb{R}^d)$ to denote the space of quasiperiodic functions on $\mathbb{R}^d$ that have continuous derivatives up to order $s$. The $\mathcal{C}^s_{QP}(\mathbb{R}^d)$-norm of $f\in \mathcal{C}^\infty_{QP}(\mathbb{R}^d)$ is given by
\begin{equation*}
    \|f\|_{\mathcal{C}^s_{QP}(\mathbb{R}^d)}=\sum_{|p|\leq s}\sup_{\bm{x}\in\mathbb{R}^d}|\partial_{\bm{x}}^pf|.
\end{equation*}

Define the quasiperiodic Hilbert space: $\mathcal{L}_{QP}^2(\mathbb{R}^d) = \bigg\{f(\bm{x})\in\mathrm{QP}(\mathbb{R}^d):\displaystyle\bbint |f(\bm{x})|^2d\bm{x}<\infty\bigg\}$
equipped with the inner product and the norm
$$
    (f_1, f_2)_{\mathcal{L}^2_{QP}(\mathbb{R}^d)}=\bbint f_1(\bm{x})\bar{f}_2(\bm{x})d\bm{x},\quad \|f\|_{\mathcal{L}^2_{QP}(\mathbb{R}^d)}=
(f, f)_{\mathcal{L}^2_{QP}(\mathbb{R}^d)}^{1/2}.
$$

Then we have the Parseval's identity with respect to the $\mathcal{L}_{QP}^2(\mathbb{R}^d)$-norm,
\begin{align}\label{eqn:parseval}
    \|f\|_{\mathcal{L}^2_{QP}(\mathbb{R}^d)}^2 =  \mathcal{M}\left\{ |f(\bm{x})|^2\right\} = \sum_{\bm{k}\in\mathbb{Z}^n}|\hat{f}_{\bm{\lambda_k}}|^2.
\end{align}

Furthermore, for any $s\in\mathbb{N}^+,$ the $s$-derivative quasiperiodic Sobolev space $\mathcal{H}_{QP}^{s}(\mathbb{R}^d)$ is defined through the following inner product 
\begin{equation*}
    (f_1, f_2)_{\mathcal{H}_{QP}^{s}(\mathbb{R}^d)}=\sum_{|p|\leq s}(\partial_{\bm{x}}^p f_1,\partial_{\bm{x}}^p f_2)_{\mathcal{L}_{QP}^{2}(\mathbb{R}^d)},
\end{equation*}
and the corresponding norm
\begin{equation*}
    \|f\|^2_{\mathcal{H}_{QP}^{s}(\mathbb{R}^d)}=(f, f)_{\mathcal{H}_{QP}^{s}(\mathbb{R}^d)}^{1/2}=\sum_{\bm{\lambda}\in\Lambda}(1+|\bm{\lambda}|^2)^s|\hat{f}_{\bm{\lambda}}|^2.
\end{equation*}
The $\mathcal{H}^{s}_{QP}(\mathbb{R}^d)$ semi-norm is defined as
$$
|f|_{\mathcal{H}^{s}_{QP}(\mathbb{R}^d)}^2 = \sum_{|p|=s}(\partial_{\bm{x}}^p f,\partial_{\bm{x}}^p f)_{\mathcal{L}_{QP}^{2}(\mathbb{R}^d)} = \sum_{\bm{\lambda}\in\Lambda}|\bm{\lambda}|^{2s}|\hat{f}_{\bm{\lambda}}|^2.
$$
For the case  $s=0$, we have $\mathcal{H}_{QP}^0(\mathbb{R}^d)=\mathcal{L}^2_{QP}(\mathbb{R}^d)$. To simplify notation, we use $\|\cdot \|_0$ to denote the norm $\|\cdot\|_{\mathcal{L}^2_{QP}(\mathbb{R}^d)}$,  $\|\cdot\|_s$ to denote the norm $\|\cdot\|_{\mathcal{H}_{QP}^{s}(\mathbb{R}^d)}$.

\subsection{Quasiperiodic elliptic equations}
\label{sec:quasiperiodic_elliptic:equation}
With the preparation in the previous subsection, we introduce the quasiperiodic elliptic equation and discuss its well-posedness. Consider the following equation
\begin{align}
	\mathcal{L}u(\bm{x})=f(\bm{x}),~~\bm{x}\in\mathbb{R}^{d},
\label{eqn:elliptic}
\end{align}
where the second order elliptic operator  $\mathcal{L}: \mathcal{H}^2_{QP}(\mathbb{R}^d)\to \mathcal{L}^2_{QP}(\mathbb{R}^d)$ is defined as  
$$ \mathcal{L}u(\bm{x})= -\mathrm{div}(\alpha(\bm{x})\nabla u(\bm{x})).$$  Here, the coefficient $\alpha(\bm{x})\in \mathcal{C}_{QP}^{1}(\mathbb{R}^d)$ and the source term $f(\bm{x})\in \mathcal{L}_{QP}^{2}(\mathbb{R}^d)$. Furthermore, $\alpha(\bm{x})$ is uniformly elliptic, that is, for all $\bm{x},~\xi \in\mathbb{R}^d$, 
$$\xi^T\alpha(\bm{x})\xi\geq \gamma_0|\xi|^2, \quad \gamma_0>0.
$$
By multiplying \eqref{eqn:elliptic} with an arbitrary test function $v\in \mathcal{H}_{QP}^1(\mathbb{R}^d)$ and integrating over $\mathbb{R}^d$, we can apply an integration by parts, taking into account that the boundary integral vanishes (see \cite[Lemma 3.4]{jiang2024high}). This leads to the  variational formulation associated with \eqref{eqn:elliptic}: seek $u\in\mathcal{H}^{1}_{QP}(\mathbb{R}^d)$ such that
\begin{equation}\label{eqn:variational}
    \mathscr{B}(u, v)=(f, v), \quad\forall v\in\mathcal{H}^{1}_{QP}(\mathbb{R}^d),
\end{equation}
where the bilinear form $\mathcal{B}(\cdot, \cdot): \mathcal{H}^{1}_{QP}(\mathbb{R}^d)\times \mathcal{H}^{1}_{QP}(\mathbb{R}^d)\to\mathbb{R}$ is defined as
$$ 
{\mathscr{B}(u, v)}:= (\alpha\nabla u, \nabla v),\quad u, v\in\mathcal{H}^{1}_{QP}(\mathbb{R}^d).
$$ 
The solution to \eqref{eqn:variational} is not unique and can be determined up to a constant. To address the non-uniqueness of the solution, we impose the assumption that $u$ has zero mean value. This allows us to define the space 
$$
\overline{\mathcal{H}}_{QP}^1(\mathbb{R}^d):=\Big\{u(\bm{x})\in\mathcal{H}_{QP}^1(\mathbb{R}^d): ~\bbint u(\bm{x})d\bm{x}=0\Big\}
$$
equipped with $\mathcal{H}_{QP}^1(\mathbb{R}^d)$-norm. The revised variational problem for \eqref{eqn:elliptic} is now as follows: find $u\in\overline{\mathcal{H}}^{1}_{QP}(\mathbb{R}^d)$ such that
\begin{equation}\label{eqn:variational_1}
    {\mathscr{B}(u, v)}=(f, v), \quad\forall v\in\overline{\mathcal{H}}^{1}_{QP}(\mathbb{R}^d).
\end{equation}
By applying the Lax-Milgram Lemma,  the existence and  uniqueness of the weak solution  to \eqref{eqn:variational_1} in $\overline{\mathcal{H}}^1_{QP}(\mathbb{R}^d)$ follow immediately.

\begin{theorem}\label{thm:well_posed}
The variation problem \eqref{eqn:variational_1} has a unique solution.
\end{theorem}

Before proving \Cref{thm:well_posed}, we introduce two essential results: the equivalent modulus theorem (\Cref{lemma_appendix:eqv_mod_thm}) and  the Poincar\'e inequality (\Cref{lemma_appendix:Poincaré's inequality}), both crucial for our proof.

\begin{lemma}[Equivalent modulus theorem \cite{braess2001finite}]\label{lemma_appendix:eqv_mod_thm}
If bounded linear functionals $L_1, L_2,\cdots,L_K$ on $\mathcal{H}_{QP}^s(\mathbb{R}^d), s\in\mathbb{N}^{+}$  is not simultaneously zero for any non-zero polynomial of degree $\leq s-1$, then for all $u\in \mathcal{H}_{QP}^s(\mathbb{R}^d)$, modules $\|u\|_s$ and $|u|_s+\sum_{i=1}^K|L_i(u)|$ are equivalent, that is, there are constants $c_1, c_2>0$ such that the following inequality holds
\begin{equation}\label{eqn:eq_mod_thm}
c_1\bigg(|u|_s+\sum_{i=1}^K|L_i(u)|\bigg)\leq\|u\|_s\leq c_2 \bigg(|u|_s+\sum_{i=1}^K|L_i(u)|\bigg).
\end{equation}
\end{lemma}

\begin{lemma}[Poincar\'e inequality]\label{lemma_appendix:Poincaré's inequality}
    For all $u\in \overline{\mathcal{H}}^1_{QP}(\mathbb{R}^d)$, we have
    \begin{equation}\label{eqn_appendix:poincare}
        \|u\|_1\leq c_2|u|_1.
    \end{equation}
\end{lemma}

The proofs of these two lemmas are provided in \Cref{sec_appendix:eqv_mod_thm} and \Cref{sec_appendix:poincaré_inequality}, respectively. 
Now we prove the \Cref{thm:well_posed}.
\begin{proof}
According to the Lax-Milgram theorem, 
we need to prove the boundedness and coercivity of the bilinear form $\alpha(\cdot, \cdot)$. The upper boundedness can be directly obtained from 
\begin{align}
    |\mathscr{B}(u, v)|\leq\gamma_1\Big|\bbint_{\mathbb{R}^d}\nabla u \overline{\nabla v}d\bm{x}\Big|\lesssim|u|_1|v|_1\lesssim\|u\|_1\|v\|_1.
\end{align}
Using the Poincar\'e inequality \eqref{lemma_appendix:Poincaré's inequality}, we have
\begin{align}
|\mathscr{B}(u,u)|\geq\gamma_0\Big|\bbint_{\mathbb{R}^d}\nabla u\overline{\nabla u}d\bm{x}\Big|\gtrsim |u|_1|u|_1\gtrsim\|u\|_1^2,
\end{align}
which implies that $\mathscr{B}(\cdot, \cdot)$ is coercive. 
\end{proof}

\subsection{Projection method}
\label{sec:numerical_methods:pm}

In this subsection, we provide a detailed introduction to the PM. For an integer $N\in\mathbb{N}^{+}$, denote	 
\begin{align*}
	K_N^{n}=\left\{\bm{k}=(k_j)_{j=1}^{n}\in\mathbb{Z}^{n}:-N/2\leq k_j <N/2 \right\}.
\end{align*}
Next, we discretize the torus $\mathbb{T}^{n}$ as
$$
	\mathbb{T}_N^n=\left\{\bm{y_j}=( 2\pi j_1/N, \cdots, 2\pi j_n/N )\in\mathbb{T}^{n}: 0\leq j_1,\cdots, j_n< N, ~\bm{j}=(j_1,\cdots, j_n)\right\}.
$$
Namely, we distribute $N$ discrete points uniformly in each spatial direction of the $n$-dimensional space, resulting in  $D=N^n$ discrete points in total. We denote the grid function space as 
$$
	S_N=\left\{\mathcal{F}:\mathbb{Z}^{n}\mapsto\mathbb{C},~ \mathcal{F}\text{ is  periodic on} ~
     \mathbb{T}_N^n\right\}
$$
and define the $\ell^2$-inner product $\left(\cdot , \cdot \right)_N$ on $S_N$ as
$$
	\left(\phi_1,\phi_2\right)_N=\frac{1}{D}\sum_{\bm{y_j}\in\mathbb{T}^n_N}\phi_1(\bm{y_j})\overline{\phi}_2(\bm{y_j}),~~\phi_1,\phi_2\in S_N.
$$

For $\bm{k}_1, \bm{k}_2\in\mathbb{Z}^n$, $\bm{P}\in\mathbb{P}^{d\times n}$, we denote $\bm{Pk}_{1}{\overset{Z}{=}}\bm{Pk}_2$, if  
$\bm{Pk}_1=\bm{Pk}_2 +{Z}\bm{Pm}, {Z}\in\mathbb{Z}, \bm{m}\in\mathbb{Z}^n$. In particular, if $\bm{P}=I_n$,  $\bm{k}_1{\overset{Z}{=}}\bm{k}_2$ implies $\bm{k}_1=\bm{k}_2+{Z}\bm{m}$.
We have the following discrete orthogonality
\begin{equation}\label{eq:orth}
\Big(e^{\imath\bm{k}_1^T\bm{y_j}},e^{\imath\bm{k}_2^T\bm{y_j}}\Big)_N=
		\left\{
		\begin{aligned}
			&1, ~~~~\bm{k}_1{\overset{Z}{=}}\bm{k}_2,\\
		  &0,~~~~ \mathrm{otherwise}.\\
		\end{aligned}    			
		\right.
\end{equation}
Consequently, we can compute the discrete Fourier coefficients of a high-dimensional periodic function $F(\bm{y})$ as follows
$$
\widetilde{\mathcal{F}}_{\bm{k}}=\Big(\mathcal{F}(\bm{y_j}),e^{\imath\bm{k}^T\bm{y_j}}\Big)_N=\frac{1}{D}\sum_{\bm{y_j}\in\mathbb{T}^n_N}\mathcal{F}(\bm{y_j})e^{-\imath\bm{k}^T\bm{y_j}}.
$$

{From  \Cref{lemma:consistent_fourier_coefficient}, we have} 
\begin{equation}\label{eqn:discrete_Fourier_coefficient}
\widetilde{f}_{\bm{\lambda_k}}=\widetilde{\mathcal{F}}_{\bm{k}},
\end{equation}
which allows us to numerically simulate low-dimensional  quasiperiodic functions  using information from high-dimensional parent functions.
Therefore, we can define the \emph{discrete Fourier-Bohr transform} of a quasiperiodic function $f(\bm{x})$ as follows
\begin{equation}\label{eqn:dft}  f(\bm{x_j})=\sum_{\bm{\lambda_k}\in\Lambda_{N}^{d}}\widetilde{f}_{\bm{\lambda_k}}e^{\imath\bm{\lambda_k}^T\bm{x_j}},\quad \Lambda_N^d=\left\{\bm{\lambda_k}: \bm{\lambda_k}=\bm{Pk}, ~\bm{k}\in K_N^n\right\},
\end{equation}
where $\bm{x_j}\in X_{\bm{P}}:=\left\{\bm{x_j}=\bm{Py_j}: ~\bm{y_j}\in\mathbb{T}_N^n,~ \bm{P}\in\mathbb{P}^{d\times n}\right\}$ are \emph{collocation points}. The truncation  of a quasiperiodic function $f$  defined as 
\begin{equation*}
        T_N f(\bm{x})=\sum_{\bm{\lambda}_{\bm{k}}\in\Lambda_N^d} \hat{f}_{\bm{\lambda_k}}e^{\imath\bm{\lambda}_{\bm{k}}^T\bm{x}}:=f_T(\bm{x}),
\end{equation*}
 and the trigonometric interpolation of $f$ is given by 
$$
I_Nf(\bm{x})=\sum_{\bm{\lambda_k}\in\Lambda_{N}^{d}}\widetilde{f}_{\bm{\lambda_k}}e^{\imath\bm{\lambda_k}^T\bm{x}}:=f_I(\bm{x}).
$$ 
Consequently, $f(\bm{x}_j)=I_Nf(\bm{x_j}).$ 


\section{PM for quasiperiodic elliptic PDEs: formulation and analysis}
\label{sec:numerical_methods} 

In this section, we  present the PM for the numerical solution of quasiperiodic elliptic equations.
We derive the corresponding discrete scheme  and linear system for \eqref{eqn:elliptic} by using the tensor-vector-index conversion technique. Moreover, we provide a rigorous convergence analysis for the PM method in solving quasiperiodic elliptic PDEs.


\subsection{PM formulation for quasiperiodic elliptic PDEs}
\label{sec:numerical_methods:discrete_scheme_and_linear_system}

In this subsection, we discretize equation \eqref{eqn:elliptic} using the PM. We  then develop a discrete scheme based on the variational formulation and generate the corresponding linear system using index conversion.

\subsubsection{Discrete scheme} 
\label{sec:numerical_methods:derivation}
Within the PM framework, we can discretize $\alpha(\bm{x})$ and $f(\bm{x})$ in \eqref{eqn:elliptic} as follows
	$$	\alpha(\bm{x_j})=\sum_{\bm{\lambda}_\alpha\in\Lambda_{\alpha,N}^d}\widetilde{\alpha}_{\bm{\lambda}_\alpha}e^{\imath\bm{\lambda}^T_\alpha\bm{x_j}},	\quad\widetilde{\alpha}_{\bm{\lambda}_\alpha}=\widetilde{\mathcal{A}}_{\bm{k}_\mathcal{A}}=\frac{1}{D}\sum_{\bm{y_j}\in\mathbb{T}_N^n}\mathcal{A}(\bm{y_j})e^{-\imath\bm{k}^T_\mathcal{A}\bm{y_j}},\quad \bm{x_j}\in X_{\bm{P}},
	$$

	$$
f(\bm{x_j})=\sum_{\bm{\lambda}_f\in\Lambda_{f,N}^d}\widetilde{f}_{\bm{\lambda}_f}e^{\imath\bm{\lambda}^T_f\bm{x_j}},\quad\widetilde{f}_{\bm{\lambda}_f}=\widetilde{\mathcal{F}}_{\bm{k}_{\mathcal{F}}}=\frac{1}{D}\sum_{\bm{y_j}\in\mathbb{T}_N^n}\mathcal{F}(\bm{y_j})e^{-\imath\bm{k}^T_{\mathcal{F}}\bm{y_j}},\quad \bm{x_j}\in X_{\bm{P}},
	$$
where $\mathcal{A},  \mathcal{F}$ are parent functions of $\alpha, f$, respectively. {The set $\Lambda_{\alpha,N}^d$ is defined as follows
\begin{equation*}
    \Lambda_{\alpha,N}^d=\left\{\bm{\lambda}_\alpha: \bm{\lambda}_\alpha=\bm{Pk}_{\mathcal{A}}, ~\bm{k}_{\mathcal{A}}\in K_N^n\right\}.
\end{equation*}
Similarly, the set $\Lambda_{f, N}^d$ is defined by  
\begin{equation*}
\Lambda_{f,N}^d=\left\{\bm{\lambda}_f: \bm{\lambda}_f=\bm{Pk}_{\mathcal{F}}, ~\bm{k}_{\mathcal{F}}\in K_N^n\right\}.    
\end{equation*}
}
A simple calculation reveals that the spectral points of $u$ satisfy $\bm{\lambda}_u=\bm{\lambda}_f-\bm{\lambda}_\alpha$, which means that the spectral point set of $u$ is spanned by the spectral point sets of $\alpha$ and $f$. Based on this, we can determine the set
 $$
 \Lambda_{u,N}^d\subseteq \mathrm{span}_{\mathbb{Z}}\left\{\Lambda_{\alpha, N}^d,\Lambda_{f, N}^d\right\}:=\left\{k_1\bm{\lambda}_\alpha+k_2\bm{\lambda}_f: \bm{\lambda}_\alpha\in\Lambda_{\alpha, N}^d, ~\bm{\lambda}_f\in\Lambda_{f, N}^d,~ k_1, k_2\in\mathbb{Z}\right\}
 $$ 
 and  discretize $u(\bm{x})$ as
$$
u(\bm{x_j})=\sum_{\bm{\lambda}_u\in\Lambda_{u,N}^d}\widetilde{u}_{\bm{\lambda}_u}e^{\imath\bm{\lambda}^T_u\bm{x_j}}, \quad\widetilde{u}_{\bm{\lambda}_u}=\widetilde{\mathcal{U}}_{\bm{k}_\mathcal{U}}=\frac{1}{D}\sum_{\bm{y_j}\in\mathbb{T}_N^n}\mathcal{U}(\bm{y_j})e^{-\imath\bm{k}^T_\mathcal{U}\bm{y_j}},\quad\bm{x_j}\in X_{\bm{P}},
$$
where $\mathcal{U}$ is the parent function of $u$.
\begin{remark}
We provide a simple example to illustrate how the spectral point set $\Lambda_{u,N}^d$ is spanned. Let $\Lambda_{\alpha,N}^d\subseteq\mathbb{Z}(1,\sqrt{2})$ and $\Lambda_{f,N}^d\subseteq\mathbb{Z}(1,\sqrt{3})$, where $\mathbb{Z}(1,\sqrt{2})$ represents the spectral point set formed by integer linear combinations of 1 and $\sqrt{2}$, and similarly for $\mathbb{Z}(1,\sqrt{3})$. In this case, $\Lambda_{u,N}^d$ can be spanned by these two spectral point sets, which can be expressed as follows  
$$
    \Lambda^d_{u,N}\subseteq\mathrm{span}_{\mathbb{Z}}\left\{\Lambda_{\alpha,N}^d,\Lambda_{f,N}^d\right\}=\mathbb{Z}(1,\sqrt{2},\sqrt{3}).
$$
\end{remark}

 Let $ V^N:=\mathrm{span}\big\{e^{\imath\bm{\lambda}_v^T\bm{x}}:\bm{\lambda}_v\in\Lambda_{u,N}^d,~\bm{x}\in\mathbb{R}^{d}\}$ be a finite-dimensional subspace of  $\mathrm{QP}(\mathbb{R}^d)$. We then select a test function $v$ from the space $\overline{V}^N:=\left\{ v\in V^N:~ \widetilde{v}_{\bm{0}}=0\right\}\subseteq\overline{\mathcal{H}}_{QP}^1(\mathbb{R}^d)$. In the following, our goal is to find an approximate solution of \eqref{eqn:variational_1} in $\overline{V}^N$, i.e., $u\in \overline{V}^N$, such  that
 \begin{equation}\label{eqn:discrete_variational}
     {\mathscr{B}_N(u,v)}:=(\alpha\nabla u, \nabla v)_N=(f, v)_N, \quad \forall v\in \overline{V}^N.
 \end{equation}
Clearly, ${\mathscr{B}_N}(\cdot, \cdot)$ is  uniformly elliptic, thus \eqref{eqn:discrete_variational} has a unique solution in $\overline{V}^N$, , as demonstrated in \Cref{thm:pm_wellposed}.
\begin{theorem}\label{thm:pm_wellposed}
    The discrete variational formulation \eqref{eqn:discrete_variational} has a unique solution.
 \end{theorem}  
\begin{proof}
   According to the Lax-Milgram Lemma, it is sufficient to prove the boundedness and coercivity of  $\mathscr{B}_N(\cdot,\cdot)$. For all $u, v\in \overline{V}^N$, we have
    \begin{equation*}
       |\mathscr{B}_N(u,v)|\leq\gamma_1(\nabla u, \nabla v)_N=\gamma_1(\nabla u, \nabla v)\leq\gamma_1\|u\|_1\|v\|_1,
    \end{equation*}
    which implies $\mathscr{B}_N(\cdot,\cdot)$ is bounded.
    
   The coercivity of $\mathscr{B}_N(\cdot,\cdot)$ follows from
     \begin{equation*}
       |\mathscr{B}_N(u,u)|\geq \gamma_0(\nabla u, \nabla u)_N =\gamma_0(\nabla u, \nabla u)\geq\gamma_0\|u\|_1^2.
    \end{equation*}
\end{proof}

In view of the discrete orthogonality of quasiperiodic base functions   
\begin{equation}\label{eqn:orthogonality}	\Big(e^{\imath\bm{\lambda}_1^T\bm{x_j}},e^{\imath\bm{\lambda}_2^T\bm{x_j}}\Big)_N=
		\left\{
		\begin{aligned}
			&1, ~~~~\bm{\lambda}_1\overset{N}{=}\bm{\lambda}_2, \\
		  &0,~~~~ \mathrm{otherwise},\\
		\end{aligned}    			
		\right.
	\end{equation}
we have
$$
(\alpha\nabla u,\nabla v)_N=\sum_{\bm{\lambda}_u\in\Lambda_{u,N}^d}\widetilde{\alpha}_{\bm{\lambda}_v-\bm{\lambda}_u}(\bm{\lambda}_v^T\bm{\lambda}_u)\widetilde{u}_{\bm{\lambda_u}},
$$
$$ 
(f,v)_N=\widetilde{f}_{\bm{\lambda}_v}.
$$
Combining with \eqref{eqn:discrete_Fourier_coefficient}, we can derive  the discrete scheme for \eqref{eqn:elliptic} as follows
\begin{equation}\label{eqn:n-discrete_equation}
\sum_{\bm{k}_{\mathcal{U}}\in K_{N}^n}\widetilde{\mathcal{A}}_{\bm{k}_{\mathcal{A}}}(\bm{P}\bm{k}_{\mathcal{V}})^T(\bm{P}\bm{k}_{\mathcal{U}})\widetilde{\mathcal{U}}_{\bm{k}_{\mathcal{U}}}=\widetilde{\mathcal{F}}_{\bm{k}_{\mathcal{V}}},\quad \bm{k}_{\mathcal{A}}=\bm{k}_{\mathcal{V}}\overset{N}{-}\bm{k}_{\mathcal{U}},
\end{equation}
where $\bm{k}_{\mathcal{V}}\overset{N}{-}\bm{k}_{\mathcal{U}}:=(\bm{k}_{\mathcal{V}}-\bm{k}_{\mathcal{U}}) (\mathrm{mod}~N)$. Note that for a given vector $\bm{k}=(k_1, \cdots, k_n)\in\mathbb{Z}^n$, we have $\bm{k}(\mathrm{mod}~N)\equiv (k_1(\mathrm{mod}~N),\cdots, k_n(\mathrm{mod}~N)).$

\subsubsection{Linear system}

The discrete scheme \eqref{eqn:n-discrete_equation} is formulated with respect to the tensor index $\bm{k}=(k_1, \cdots, k_n)\in K_N^n$. To solve it numerically, we introduce a general index conversion between tensors and vectors to obtain a matrix linear system.

For a given tensor $\mathcal{I}\in\mathbb{C}^{N_1\times\cdots\times N_n}$, we denote its size vector as
$$
\bm{N}=(N_k)_{k=1}^n, \quad N_k\in\mathbb{N}^+
$$
and the set
$$
K_{\bm{N}}^n=\left\{\bm{i}=(i_k)_{k=1}^n\in\mathbb{Z}^n: -N_k/2\leq i_k< N_k/2\right\}.
$$
Let $\bm{i}=(i_1, \cdots, i_n)\in K_{\bm{N}}^n$ be the index of $\mathcal{I}$. Define the \emph{convert bijection} as follows
\begin{equation}\label{eqn:conversion}
     \begin{aligned}
    \mathcal{C}: \quad&K_{\bm{N}}^n \rightarrow \mathbb{N},\\
	&\bm{i}\xrightarrow{\mathcal{C}} i, 
     \end{aligned}
\end{equation}
 where $i$ is determined by the rule
\begin{equation}
    i=\sum_{k=1}^{n} \overline{i}_k \prod_{t=k+1}^n N_t,\quad \overline{i}_k:= i_k(\mathrm{mod}~N_k).
\end{equation}
The index inverse conversion $\mathcal{C}^{-1}$ is defined by
\begin{equation}\label{eqn:inverse_convert}	i_k=
		\left\{
		\begin{aligned}
			&\overline{i}_k, ~~~~0\leq\overline{i}_k< N_k/2, \\
		  &\overline{i}_k-N_k,~~~~ N_k/2\leq\overline{i}_k< N_k,\\
		\end{aligned}    			
		\right.
  \quad \overline{i}_{k}=\Big\lfloor i~ (\mathrm{mod}~\prod_{t=k}^{n}N_t) \displaystyle/ \prod_{t=k+1}^n N_t \Big\rfloor,
	\end{equation}
where  $\lfloor \cdot\rfloor$ is the  floor symbol. We refer to the bijection $\mathcal{C}$ as the tensor-vector-index conversion. 

{To more clearly explain the process of index conversion, we provide a simple example. Consider a 2-dimensional tensor with $N_1 = N_2 =4$ (so $D=16$). The tensor indices $\bm{i}=(i_1, \cdots, i_n)$ range from $(-2, -2)$ to $(1, 1)$. Consider the tensor index $\bm{i}=(-1, 0)$.
\begin{itemize}
    \item \textbf{Convert to vector index:} Compute $\overline{i}_1, \overline{i}_2$
    $$
        \overline{i}_1=-1~ \mathrm{mod}~ 4 =3,\quad \overline{i}_2=0~ \mathrm{mod}~ 4 =0.
    $$
    Then the mapped vector index $i$ can be calculated by
    $$
    i=\overline{i}_1\cdot N+\overline{i}_2=12.
    $$
    \item \textbf{Inverse conversion:} For the vector index $i=12$, compute $\overline{i}_1, \overline{i}_2$
    $$
     \overline{i}_1=\lfloor12 ~\mathrm{mod}~(4\cdot4)\rfloor/4=3,\quad  \overline{i}_2=\lfloor0 ~\mathrm{mod}~(4)\rfloor/1=0.
    $$
Since $\overline{i}_1=3$  and $N_1 = 4$, the original value is $i_1=3-4=-1$. Thus, the inverse conversion correctly gives $\bm{i}=(-1, 0)$.
\end{itemize}}

{Next, we generate the linear system using $\eqref{eqn:conversion}$ with   $N_k=N$ for each $k$, since under the PM framework, the degree of freedom in each dimension is $N$.} Define
$$
A=(A_{ij})\in\mathbb{C}^{D\times D},\quad A_{ij}=\widetilde{\mathcal{A}}_{\bm{k}_{\mathcal{V}}\overset{N}{-}\bm{k}_{\mathcal{U}}},
$$
$$
W=(W_{ij})\in\mathbb{C}^{D\times D},\quad W_{ij}=(\bm{P}\bm{k}_{\mathcal{V}})^T(\bm{P}\bm{k}_{\mathcal{U}}),
$$
where indices $i, j$ are determined by
\begin{equation}\label{eqn:ij}
    \bm{k}_{\mathcal{V}}\xrightarrow{\mathcal{C}} i, \quad \bm{k}_{\mathcal{U}}\xrightarrow{\mathcal{C}} j,
\end{equation}
respectively. And the column vectors $\bm{U}$ and $\bm{F}$ are defined, respectively, by
$$
\bm{U}=(U_j)\in\mathbb{C}^{D},\quad U_j = \widetilde{\mathcal{U}}_{\bm{k}_{\mathcal{U}}},
$$
$$
\bm{F}=(F_i)\in\mathbb{C}^{D},\quad F_i = \widetilde{\mathcal{F}}_{\bm{k}_{\mathcal{V}}}.
$$
Consequently, we obtain the following  linear system 
\begin{equation}\label{eqn:linear_system}
Q\bm{U}=\bm{F},\quad Q=A\circ W\in\mathbb{C}^{D\times D}.   
\end{equation}

\begin{remark}
As  $N$ and $n$ increase,  the size of the matrix $Q\in\mathbb{C}^{D\times D}$, {where $D=N^n$,} grows excessively large, leading to a substantial increase in computational cost and memory requirements. Meanwhile, as the size of $Q$ increases, the linear system becomes ill-conditioned, slowing convergence with standard iterative methods.
To address these challenges, in \Cref{sec:compressed_storage}, we will propose a compressed storage method for reducing memory usage and an efficient preconditioner in solving the linear system \eqref{eqn:linear_system}. 

\end{remark}


\subsection{Convergence analysis}
\label{sec:convergence_analysis}
In this subsection, we provide a thorough convergence analysis of PM for solving the quasiperiodic elliptic equation \eqref{eqn:elliptic}.  

\subsubsection{Main theorem}
Let $u$ be the solution of problem \eqref{eqn:elliptic},   and $u_N$ the solution of the corresponding discrete problem
{\begin{equation}\label{eqn:discrete_elliptic}
    (\widetilde{\mathcal{L}}u, v)= (f_I, v),\quad v\in\overline{V}^N,
\end{equation}}
where the operator $\widetilde{\mathcal{L}}$  is defined by
\begin{equation}\label{eqn:interpolation_elliptic}
    \widetilde{\mathcal{L}}u:=-\mathrm{div}\left( I_N(\alpha\nabla u)\right).
\end{equation}
We now present the main result of error analysis.

{\begin{theorem}\label{thm:convergence_analysis}
Let $s\geq1$. Assume
\begin{itemize}
    \item[1.]  $u\in \overline{\mathcal{H}}_{QP}^{1}(\mathbb{R}^d)$ is the solution to  equation \eqref{eqn:variational_1}, with  parent function $\mathcal{U}\in\mathcal{H}^{s+1}(\mathbb{T}^n)$.
    \item[2.] The coefficient $\alpha\in\mathcal{C}^1_{QP}(\mathbb{R}^d)$ in  equation \eqref{eqn:variational_1} is uniformly elliptic and bounded, i.e., for all $\bm{x},~\xi \in\mathbb{R}^d$, 
$$\gamma_0|\xi|^2\leq\xi^T\alpha(\bm{x})\xi\leq\gamma_1|\xi|^2,\quad \gamma_0, \gamma_1>0.
$$  Moreover, its parent function satisfies $\mathcal{A}\in\mathcal{H}^{s}(\mathbb{T}^n)$.
    \item[3.] The source term $f$ for the equation \eqref{eqn:variational_1} belongs to $\mathcal{L}^2_{QP}(\mathbb{R}^d)$ and its parent function $\mathcal{F}\in \mathcal{H}^{s}(\mathbb{T}^n)$.
    \item[4.]  There exists a constant $C^*$, independent of $\mathcal{F}$, such that  the following regularity estimate holds
    \begin{equation}\label{eqn:regularity}
\|\mathcal{U}\|_{\mathcal{H}^{s+1}(\mathbb{T}^n)}\leq C^*\|\mathcal{F}\|_{\mathcal{H}^{s}(\mathbb{T}^n)}.    
\end{equation}
\end{itemize}
Then, there exists a constant $C$, independent of $\mathcal{F}$ and $N$,  such that the error estimate
$$
\|u-u_N\|_{0}\leq CN^{-s}\|\mathcal{F}\|_{\mathcal{H}^{s}(\mathbb{T}^n)}
$$ 
holds.
\end{theorem}}
\textbf{Sketch of proof.}  Using the triangle inequality, we can decompose the numerical error $\|u-u_N\|_0$ into two components: the \emph{truncation error} $\|u-u_T\|_0$ and the \emph{aliasing error} $\|u_T-u_N\|_0$. Here $u_T$ refers to the truncation of the solution $u$ to the \eqref{eqn:elliptic}. The truncation error can be derived from  \Cref{lemma:truncation}  in \Cref{sec:convergece_results}. Subsequently, we establish an upper bound for the aliasing error in \Cref{thm:wn}. Finally, the proof of the main result is completed in  \Cref{sec:convergence_proof}.

\subsubsection{Bounds for truncation and aliasing errors}\label{sec:convergece_results}
Before proving  the main conclusion,  we  present some essential results. \Cref{lemma:truncation}  provides the truncation error estimate of quasipeirodic functions and a rigorous convergence result for the PM, as also presented in  \cite[Theorem 5.1]{jiang2024numerical} and \cite[Theorem 5.3]{jiang2024numerical}, respectively.

{\begin{lemma}\label{lemma:truncation}
For $s\geq0$,  we assume that $f(\bm{x})\in\mathrm{QP}(\mathbb{R}^d)$ and its parent function $\mathcal{F}(\bm{y})\in \mathcal{H}^s(\mathbb{T}^n)$. Then, there exist  constants $C_1, C_2$, independent of $\mathcal{F}$ and $N$, such that
\begin{align}
\label{eqn:truncation}
    \|T_Nf-f\|_{0}\leq C_1N^{-s}|\mathcal{F}|_{\mathcal{H}^s(\mathbb{T}^n)},
    \\
\label{eqn:interpolation}
     \|I_Nf-f\|_0\leq C_2N^{-s}|\mathcal{F}|_{\mathcal{H}^s(\mathbb{T}^n)}.
\end{align}
 \end{lemma}}

Using \Cref{lemma:truncation}, we can directly derive the estimates for the truncation and interpolation errors of the gradient of quasiperiodic functions, as presented in \Cref{lemma:nabla}.
{\begin{lemma}\label{lemma:nabla}
Let $s\geq 1$. Assume $u\in \mathcal{H}^1_{QP}(\bbR^d), \alpha\in\mathcal{C}^1_{QP}(\mathbb{R}^d)$ and their parent functions $\mathcal{U}\in {\mathcal{H}}^{s+1}(\mathbb{T}^n), \mathcal{A}\in{\mathcal{H}}^{s}(\mathbb{T}^n)$, respectively. Under these assumptions, there exists constants $C'_1, C'_2$, independent of $\mathcal{U}$ and $N$, such that the following inequalities hold
\begin{align}
\label{eqn:nabla1}
    \|T_N(\alpha\nabla u)-\alpha\nabla u \|_{0}\leq C'_1N^{-s}\|\mathcal{U}\|_{\mathcal{H}^{s+1}(\mathbb{T}^n)},
    \\
\label{eqn:nabla2}
 \|I_N(\alpha\nabla u)-\alpha\nabla u \|_{0}\leq C'_2N^{-s}\|\mathcal{U}\|_{\mathcal{H}^{s+1}(\mathbb{T}^n)}.
\end{align}
\end{lemma}}
\begin{proof}
    
where $\widetilde{\nabla}$ is given by $\widetilde{\nabla}=\left(\sum\limits_{i=1}^d p_{i1}\dfrac{\partial}{\partial y_{1}}, \cdots, \sum\limits_{i=1}^d p_{in}\dfrac{\partial}{\partial y_{n}}\right).$
Moreover, since $\mathcal{A}(\bm{y})\in\mathcal{H}^{s}(\mathbb{T}^n)$, there exist a constant $C_{\mathcal{A}}$ such that
$$
|\mathcal{A}|_{\mathcal{H}^s(\mathbb{T}^n)}\leq C_{\mathcal{A}}
$$
holds. Combining this with Lemma 3.2,  we have
\begin{equation}\nonumber
    \begin{aligned}
        \|T_N\phi-\phi\|_{0}
        &\leq C_1N^{-s}\|\Phi\|_{\mathcal{H}^{s}(\mathbb{T}^n)}\\
        &\leq C_1N^{-s}|\mathcal{A}|_{\mathcal{H}^{s}(\mathbb{T}^n)}\|\widetilde{\nabla} \mathcal{U}\|_{\mathcal{H}^{s}(\mathbb{T}^n)}\\
        &\leq C_1N^{-s}C_{\mathcal{A}}\|\widetilde{\nabla} \mathcal{U}\|_{\mathcal{H}^{s}(\mathbb{T}^n)}\\
        &\leq C^*_1 N^{-s}\|\bm{P}\|_\infty\|{\nabla} \mathcal{U}\|_{\mathcal{H}^{s}(\mathbb{T}^n)}\quad\quad (C^*_1=C_1C_{\mathcal{A}})\\
        &\leq C
        _1'N^{-s}\|\mathcal{U}\|_{{\mathcal{H}^{s+1}(\mathbb{T}^n)}}\quad\quad (C'
        _1=C^*_1\|\bm{P}\|_{\infty}),
    \end{aligned}
\end{equation}
where $C'_1, C^*_1, C_1, C_{\mathcal{A}}$ are positive constants, and $\nabla\mathcal{U}=\left(\dfrac{\partial\mathcal{U}}{\partial y_1},\cdots, \dfrac{\partial\mathcal{U}}{\partial y_n}\right)$. This implies that inequality \eqref{eqn:nabla1} holds.\\
Similarly, by using  Lemma 3.2, we have the following estimate
\begin{equation}\nonumber
        \|I_N\phi-\phi\|_{0}\leq C'_2N^{-s}\|\mathcal{U}\|_{{\mathcal{H}^{s+1}(\mathbb{T}^n)}},
\end{equation}
which implies that \eqref{eqn:nabla2} holds. 
\end{proof}

Using \Cref{lemma:truncation} and \Cref{lemma:nabla}, we estimate the upper bound for  $\|u_T-u_N\|_0$ as follows.
{\begin{lemma}\label{thm:wn}
Let $s\geq 1$. Assuming the same conditions as stated in \Cref{thm:convergence_analysis},  we have the following result, there exists a constant $C'$, independent of $\mathcal{F}$ and $N$,  such that the following estimate holds
$$
\|u_T(\bm{x})-u_N(\bm{x})\|_0\leq C'N^{-s}\|\mathcal{F}\|_{\mathcal{H}^{s}(\mathbb{T}^n)}.
$$
\end{lemma}}
\begin{proof}
Define $\psi(\bm{x})=u(\bm{x})-u_T(\bm{x})$ and $w_N(\bm{x})=u_T(\bm{x})-u_N(\bm{x})$. We can rewrite  \eqref{eqn:discrete_elliptic} as 
$$
(\widetilde{\mathcal{L}}w_N, w_N)=((\widetilde{\mathcal{L}}-\mathcal{L})u_T, w_N)+(\mathcal{L}\psi, w_N)+ (f-f_I, w_N),
$$
and have the subsequent inequality 
\begin{equation}\nonumber
	\begin{aligned}
		\gamma_0\|w_N\|_1^2\leq(\widetilde{\mathcal{L}} w_N,w_N)&=((\widetilde{\mathcal{L}}-\mathcal{L})u_T,w_N)+(\mathcal{L}\psi,w_N)+(f-f_I,w_N)\\
		&=T_1+T_2+T_3,
	\end{aligned}
\end{equation}
using the ellipticity of $\widetilde{\mathcal{L}}$. Next, we estimate the  bounds for $T_1, T_2, T_3$, respectively.

(i) For $T_1,$ we set $\phi=\alpha\nabla u_T\in \mathrm{QP}(\mathbb{R}^d).$ Combining with \eqref{eqn:nabla2}, we obtain the  inequality
$$
T_1=(\phi-I_N\phi,\nabla w_N)\leq {C_2}N^{-s}\|\mathcal{U}\|_{\mathcal{H}^{s+1}(\mathbb{T}^n)}\|\nabla w_N\|_0\leq {C_2} N^{-s}\|\mathcal{U}\|_{\mathcal{H}^{s+1}(\mathbb{T}^n)}\|w_N\|_1.
$$

(ii) {For $T_2$, there exists a constant $C_{\gamma}=\gamma_1C_1$ such that}
$$T_2=(\alpha\nabla \psi, \nabla w_N)
        \leq {\gamma_1}\|\nabla \psi\|_0\|\nabla w_N\|_0
		\leq {C_{\gamma}}N^{-s}\|\mathcal{U}\|_{\mathcal{H}^{s+1}(\mathbb{T}^n)}\|w_N\|_1,$$
 using the \eqref{eqn:nabla1} from  \Cref{lemma:nabla}.

(iii) For $T_3,$ we have
$$
T_3=(f-f_I, w_N)\leq \|f-I_Nf\|_{0}\|w_N\|_0\leq {C_2}N^{-s}\|\mathcal{F}\|_{\mathcal{H}^{s}(\mathbb{T}^n)}\|w_N\|_1.
$$
Combining the upper bounds of $T_1, T_2$ and $T_3$ with the regularity condition  \eqref{eqn:regularity}, we obtain 
$$
\gamma_0\|w_N\|_1^2\leq {C''} N^{-s}\|w_N\|_1\|\mathcal{F}\|_{\mathcal{H}^{s}(\mathbb{T}^n)}, 
$$
where {$C''=(C_2+C_\gamma)C^*+C_2$, independent of $\mathcal{F}$ and $N$}. This implies
$$
\|w_N\|_0\leq \|w_N\|_1\leq {C'} N^{-s}\|\mathcal{F}\|_{\mathcal{H}^{s}(\mathbb{T}^n)},
$$
{where $C'=C''/\gamma_0$ is independent of $\mathcal{F}$ and $N$.}
\end{proof}

\subsubsection{Proof of \Cref{thm:convergence_analysis}}\label{sec:convergence_proof}
Applying the triangle inequality, we have
\begin{equation*}
    \begin{aligned}
    \|u(\bm{x})-u_N(\bm{x})\|_0&\leq\|u(\bm{x})-u_T(\bm{x})\|_0+\|u_T(\bm{x})-u_N(\bm{x})\|_0\\
    &:=\|\psi(\bm{x})\|_0+\|w_N(\bm{x})\|_0.
\end{aligned}
\end{equation*}
From the \eqref{eqn:truncation} provided in \Cref{lemma:truncation}, we obtain
  $$
\|\psi\|_0=\|u-u_T\|_0\leq {C_1}N^{-s}|\mathcal{U}|_{\mathcal{H}^{s}(\mathbb{T}^n)}\leq {C_1} N^{-s}\|\mathcal{U}\|_{\mathcal{H}^{s+1}(\mathbb{T}^n)},
 $$
 As for the bound of $\|w_N\|_0$, we can conclude from the \Cref{thm:wn} that
 $$
\|w_N\|_0=\|u_T-u_N\|_0\leq {C'}N^{-s}\|\mathcal{F}\|_{\mathcal{H}^{s}(\mathbb{T}^n)}.
 $$ 
Thus, we obtain the error estimate for the projection method, 
\begin{equation*}
	\begin{aligned}
		\|u-u_N\|_0&\leq\|\psi\|_0+\|w_N\|_0\\
		&\leq {C_1} N^{-s}\|\mathcal{U}\|_{\mathcal{H}^{s+1}(\mathbb{T}^n)}+ {C'} N^{-s}\|\mathcal{F}\|_{\mathcal{H}^{s}(\mathbb{T}^n)}\\
		&\leq C N^{-s}\|\mathcal{F}\|_{\mathcal{H}^{s}(\mathbb{T}^n)},
	\end{aligned}
\end{equation*}
{where the constant $C=C_1C^*+C'$, independent of $\mathcal{F}$ and $N$.}

\begin{remark}
According to our convergence result, if the solution and the corresponding parent function are analytic (i.e., $s = +\infty$), then the convergence rate exceeds any finite power of  in $1/N$. In this case, we conclude that the PM exhibits spectral accuracy.
\end{remark}

\section{Linear system solver}
\label{sec:compressed_storage}

In this section,  we first introduce the \emph{compressed storage method} to leverage the multilevel block circulant structure of matrix $A$, reducing the memory storage requirement for $Q$.
Moreover, we propose the \emph{compressed PCG} (C-PCG) algorithm which incorporates our proposed diagonal preconditioner to accelerate the convergence rate.

 \subsection{ Multi-level block circulant  matrix}

 \label{sec:numerical_methods:matrix_structure}

In this subsection, we analyze the structure of matrix $A\in\mathbb{C}^{D\times D}$ for general case $d\leq n$. 
The structure of $A$ can be depicted by 
\begin{scriptsize}
\begin{equation*}
      \begin{gathered}
                A_{11}^{(1)}\\
                \begin{bNiceMatrix}[name=a]
                    \widetilde{\mathcal{A}}_{(\bm{0}, 0)}&\widetilde{\mathcal{A}}_{(\bm{0}, -1)}& \cdots&\widetilde{\mathcal{A}}_{(\bm{0},1)}\\
                    \widetilde{\mathcal{A}}_{(\bm{0}, 1)}&	\widetilde{\mathcal{A}}_{(\bm{0}, 0)}&\cdots &	\widetilde{\mathcal{A}}_{(\bm{0}, 2)}\\
                    \vdots&\vdots&\cdots&\vdots\\
                    \widetilde{\mathcal{A}}_{(\bm{0}, -1)}&	\widetilde{\mathcal{A}}_{(\bm{0}, -2)}& \cdots&	\widetilde{\mathcal{A}}_{(\bm{0}, 0)}
                \end{bNiceMatrix}
            \end{gathered}
            \begin{matrix}
                \\
            \longrightarrow
        \end{matrix}\ \ 
        \begin{gathered}
            A_{11}^{(2)}\\
            \begin{bNiceMatrix}[name=b]
                A_{11}^{(1)}&A_{12}^{(1)}& \cdots&A_{1N}^{(1)}\\
                A_{1N}^{(1)}&	\dbox{$A_{11}^{(1)}$}&\cdots &A_{1(N-1)}^{(1)}\\
      \vdots&\vdots&\cdots&\vdots\\
                A_{12}^{(1)}&	A_{13}^{(1)}&\cdots &A_{11}^{(1)}
            \end{bNiceMatrix}
            \end{gathered}
        \begin{matrix}
                \underset{\longrightarrow}{...}
            \end{matrix}
            \begin{gathered}
        A\\
        \begin{bNiceMatrix}[name=c]
            A_{11}^{(n-1)}&A_{12}^{(n-1)}&\cdots &A_{1N}^{(n-1)}\\
            A_{1N}^{(n-1)}&	\dbox{$A_{11}^{(n-1)}$}&\cdots &A_{1(N-1)}^{(n-1)}\\
            \vdots&\vdots&\cdots&\vdots\\
            A_{12}^{(n-1)}&	A_{13}^{(n-1)}&\cdots &A_{11}^{(n-1)}
		\end{bNiceMatrix},\end{gathered}
\end{equation*}
\end{scriptsize}
\tikz[remember picture, overlay]
{
\draw[cyan][->][color=black] ($(a-1-1)-(0.6,-0.1)$) to[out=180,in=180] ($(a-2-1)-(0.6,-0.1)$);
\draw[cyan][->][color=black] ($(a-2-1)-(0.60,0)$) to[out=180,in=180] ($(a-3-1)-(0.60,0)$);
\draw[cyan][->][color=black] ($(a-3-1)-(0.60,0)$) to[out=180,in=180] ($(a-4-1)-(0.60,0)$);
\draw[cyan][->][color=black] ($(b-1-1)-(0.40,0)$) to[out=180,in=180] ($(b-2-1)-(0.40,0)$);
\draw[cyan][->][color=black] ($(b-2-1)-(0.40,0)$) to[out=180,in=180] ($(b-3-1)-(0.40,0)$);
\draw[cyan][->][color=black] ($(b-3-1)-(0.40,0.1)$) to[out=180,in=180] ($(b-4-1)-(0.40,0)$);
\draw[cyan][->][color=black] ($(c-1-1)-(0.55,0)$) to[out=180,in=180] ($(c-2-1)-(0.55,0)$);
\draw[cyan][->][color=black] ($(c-2-1)-(0.55,0)$) to[out=180,in=180] ($(c-3-1)-(0.55,0)$);
\draw[cyan][->][color=black] ($(c-3-1)-(0.55,0.1)$) to[out=180,in=180] ($(c-4-1)-(0.55,0)$);
}
where $\bm{0}$ denotes an $(n-1)$-dimensional zero vector. 
To obtain the matrix $A$, the initial step is computing the discrete Fourier coefficients of $\mathcal{A}(\bm{x})$, denoted as $\widetilde{\mathcal{A}}$. Each row of $\widetilde{\mathcal{A}}$ is then cyclically permuted to assemble the first-level block circulant matrix $A^{(1)}$. This matrix is further permuted to recursively construct the $l$-level block circulant matrix $A^{(l)}$ ($2\leq l\leq n$), which consists of $N\times N$ blocks of $A^{(l-1)}$. As a result, $A$ is referred to as an \emph{$n$-level block circulant matrix} and can be conveniently expressed as follows
$$
A=\sum_{\bm{k}\in K_N^n}\widetilde{\mathcal{A}}_{\bm{k}}Z_{N}^{\bm{k}}=\sum_{ k_1\in K_N^1}\cdots\sum_{k_n\in K_N^1}\widetilde{\mathcal{A}}_{(k_1,\cdots,  k_n)}Z_{N}^{k_1}\otimes\cdots\otimes Z_{N}^{k_n},\quad \bm{k}=(k_1,\cdots, k_n),
$$
where
\begin{small}
\begin{equation}\nonumber
\begin{array}{cl}
Z_N^k=\begin{array}{c@{\hspace{-6pt}}l}
\left[
\begin{array}{cccccc}
0&\cdots&1&0&\cdots&0\\
\vdots&&&\ddots&&\vdots\\
\vdots&&&&\ddots&\vdots\\
1&0&&&&1\\
&\ddots&&&&\\
0&\cdots&1&\cdots&\cdots&0\\
\end{array}
\right]_{N\times N}&\begin{array}{l}
\left.\rule{0mm}{10mm}\right\}\overline{k}\\
\\
\\
\\
\end{array}
\end{array}
\\
\begin{array}{l}
\underbrace{\rule{13mm}{0mm}}_{N-\overline{k}}\quad\quad\quad\quad\qquad\quad\ \qquad
\end{array}
\end{array}
\end{equation}
\end{small}
is the \emph{cycling permutation matrix}. {Here, $\overline{k}=k\,(\mathrm{mod}~ N)$} and  $\otimes$ denotes the tensor product. It is worth mentioning that the $N$ matrices $\{Z_N^{k_i}\}$, where $i=1,\cdots,N$, can be regarded as a canonical basis for the linear space of circulant matrices of order $N$.

To enhance our intuitive grasp of matrix $A$, we provide an illustrative example of a two-level block circulant matrix for the case $d=1, n=2$. First, we examine the structures of vectors $\bm{U}$ and $\bm{F}$. Specifically, we have
$$
\bm{U}=(\bm{\mathcal{U}}_0^T~\bm{\mathcal{U}}_1^T~\cdots~\bm{\mathcal{U}}_{\frac{N}{2}-1}^T~\bm{\mathcal{U}}_{-\frac{N}{2}}^T~\cdots~\bm{\mathcal{U}}_{-1}^T)^T,
$$
where
$$
\bm{\mathcal{U}}_{m}=(\widetilde{\mathcal{U}}_{m,0}, \widetilde{\mathcal{U}}_{m,1}, \cdots, \widetilde{\mathcal{U}}_{m,\frac{N}{2}-1}, \widetilde{\mathcal{U}}_{m,-\frac{N}{2}}, \cdots, \widetilde{\mathcal{U}}_{m,-1} )^T,
$$
and

$$
\bm{F}=(\bm{\widetilde{\mathcal{F}}}_0^T~\bm{\widetilde{\mathcal{F}}}_1^T~\cdots~\bm{\widetilde{\mathcal{F}}}_{\frac{N}{2}-1}^T~\bm{\widetilde{\mathcal{F}}}_{-\frac{N}{2}}^T~\cdots~\bm{\widetilde{\mathcal{F}}}_{-1}^T)^T,
$$
where 
$$
\bm{\mathcal{F}}_{m}=(\widetilde{\mathcal{F}}_{m,0}, \widetilde{\mathcal{F}}_{m,1} ,\cdots, \widetilde{\mathcal{F}}_{m,\frac{N}{2}-1}, \widetilde{\mathcal{F}}_{m,-\frac{N}{2}}, \cdots, \widetilde{\mathcal{F}}_{m,-1} )^T.
$$
{Next, based on the discrete scheme \eqref{eqn:n-discrete_equation}, we obtain}
$$
\sum_{k_{\mathcal{U}}\in K_N^1}\widetilde{\mathcal{A}}_{m, k_{\mathcal{A}}}(\bm{Pk}_\mathcal{V})^T(\bm{Pk}_\mathcal{U})\widetilde{\mathcal{U}}_{m, k_{\mathcal{U}}}=\widetilde{\mathcal{F}}_{m, k_{\mathcal{V}}}, \quad k_{\mathcal{A}}= k_{\mathcal{V}} \overset{N}{-} k_{\mathcal{U}},
\quad \bm{k}_{\mathcal{V}}=(m,k_{\mathcal{V}})^T,\quad \bm{k}_{\mathcal{U}}=(m,k_{\mathcal{U}})^T.
$$ 
{Using the tensor-vector-index conversion, we obtain the first level circulant matrix $\bm{\mathcal{A}}_{m}$ defined by}
$$
\bm{\mathcal{A}}_{m}=\left(
\begin{array}{cccccccc}
	\widetilde{{\mathcal{A}}}_{m,0}&\widetilde{{\mathcal{A}}}_{m,-1}&\cdots&\widetilde{{\mathcal{A}}}_{m,-\frac{N}{2}}&\widetilde{{\mathcal{A}}}_{m,\frac{N}{2}-1}&\cdots&\widetilde{{\mathcal{A}}}_{m,1}\\
	\widetilde{{\mathcal{A}}}_{m,1}&\widetilde{{\mathcal{A}}}_{m,0}&\cdots&\widetilde{{\mathcal{A}}}_{m,-\frac{N}{2}+1}&\widetilde{{\mathcal{A}}}_{m,-\frac{N}{2}}&\cdots&\widetilde{{\mathcal{A}}}_{m,2}\\
	\vdots&\vdots& &\vdots&\vdots&\vdots\\
	\widetilde{{\mathcal{A}}}_{m,\frac{N}{2}-1}&\widetilde{{\mathcal{A}}}_{m,\frac{N}{2}-2}&\cdots&\widetilde{{\mathcal{A}}}_{m,0}&\widetilde{{\mathcal{A}}}_{m,-1}&\cdots&\widetilde{{\mathcal{A}}}_{m,-\frac{N}{2}}\\
 \widetilde{{\mathcal{A}}}_{m,-\frac{N}{2}}&\widetilde{{\mathcal{A}}}_{m,\frac{N}{2}-1}&\cdots&\widetilde{{\mathcal{A}}}_{m,1}&\widetilde{{\mathcal{A}}}_{m,0}&\cdots&\widetilde{{\mathcal{A}}}_{m,-\frac{N}{2}+1}\\
    \vdots&\vdots& &\vdots&\vdots&\vdots\\
    \widetilde{{\mathcal{A}}}_{m,-1}&\widetilde{{\mathcal{A}}}_{m,-2}&\cdots&\widetilde{{\mathcal{A}}}_{m,\frac{N}{2}-1}&\widetilde{{\mathcal{A}}}_{m,\frac{N}{2}-2}&\cdots&\widetilde{{\mathcal{A}}}_{m,0}
\end{array}
\right).
$$
{Furthermore, by traversing $m$ over the set $K_N^1$,  we obtain the second level block circulant matrix,}
$$
A=\left(
\begin{array}{cccccccc}
	\bm{\mathcal{A}}_{0}&\bm{\mathcal{A}}_{-1}&\cdots&\bm{\mathcal{A}}_{-\frac{N}{2}}&\bm{\mathcal{A}}_{\frac{N}{2}-1}&\cdots&\bm{\mathcal{A}}_{1}\\
	\bm{\mathcal{A}}_{1}&\bm{\mathcal{A}}_{0}&\cdots&\bm{\mathcal{A}}_{-\frac{N}{2}+1}&\bm{\mathcal{A}}_{-\frac{N}{2}}&\cdots&\bm{\mathcal{A}}_{2}\\
	\vdots&\vdots& &\vdots&\vdots&\vdots\\
	\bm{\mathcal{A}}_{\frac{N}{2}-1}&\bm{\mathcal{A}}_{\frac{N}{2}-2}&\cdots&\bm{\mathcal{A}}_{0}&\bm{\mathcal{A}}_{-1}&\cdots&\bm{\mathcal{A}}_{-\frac{N}{2}}\\
	\bm{\mathcal{A}}_{-\frac{N}{2}}&\bm{\mathcal{A}}_{\frac{N}{2}-1}&\cdots&\bm{\mathcal{A}}_{1}&\bm{\mathcal{A}}_{0}&\cdots&\bm{\mathcal{A}}_{-\frac{N}{2}+1}\\
	\vdots&\vdots& &\vdots&\vdots&\vdots\\
    \bm{\mathcal{A}}_{-1}&\bm{\mathcal{A}}_{-2}&\cdots&\bm{\mathcal{A}}_{\frac{N}{2}-1}&\bm{\mathcal{A}}_{\frac{N}{2}-2}&\cdots&\bm{\mathcal{A}}_{0}\\
\end{array}
\right).
$$
We observe that each matrix block $\bm{\mathcal{A}}_m$ consists of the discrete Fourier coefficients of $\widetilde{\mathcal{A}}_{m, k_\mathcal{A}}$ and exhibits a circulant structure. By assembling these blocks, we complete the construction of the two-level block circulant matrix $A$.

\subsection{Compressed storage}\label{sec:compressed_storage_algorithm}

In this subsection, we introduce the compressed storage method for the stiffness matrix $Q=A\circ W$. There are two steps involved in this method.

\textbf{Step 1: Compression of $A$.} 
As presented in the above subsection, 
the $n$-level block circulant matrix $A$ is recursively generated from discrete Fourier coefficients $\widetilde{\mathcal{A}}$.
This enable us to store only $\widetilde{\mathcal{A}}$ instead of the complete matrix $A$, to significantly reduce the memory usage. 
To achieve this, we utilize the relationship between the discrete Fourier coefficient $\widetilde{\mathcal{A}}_{\bm{k}_{\mathcal{A}}}$ and its corresponding position $(i ,j)$ in $A$, as determined by \eqref{eqn:ij}. By assigning the value of $\widetilde{\mathcal{A}}_{\bm{k}_\mathcal{A}}$ to $A_{ij}$, we can effectively store $D$ discrete Fourier coefficients $\widetilde{\mathcal{A}}_{\bm{k}_{\mathcal{A}}}$ along with $2D$ tensor indices $\bm{k}_{\mathcal{V}}, \bm{k}_{\mathcal{U}}$. This approach  achieves compressed storage for matrix $A$, reducing memory requirements from $\mathcal{O}(D\times D)$ to $\mathcal{O}(D)$.

\textbf{Step 2: Compression of $W$.}  
The formula for $W_{ij}$ is given by $W_{ij}=(\bm{P}\bm{k}_{\mathcal{V}})^{T}(\bm{P}\bm{k}_{\mathcal{U}})$. Since the indices $i$ and $j$ of $W_{ij}$ align with those of $A_{ij}$, we can utilize the tensor indices $\bm{k}_{\mathcal{V}}$ and $\bm{k}_{\mathcal{U}}$ associated with $A_{ij}$ to compute $W_{ij}$. This means that we can use the saved indices $\bm{k}_{\mathcal{V}}, \bm{k}_{\mathcal{U}}$ to compute $W_{ij}$ without requiring additional memory for storing $W$.

The above steps can be summarized in \Cref{alg:compressed} as the “compressed storage method" for $Q=A\circ W$.

\begin{algorithm}[H]
\renewcommand{\algorithmicrequire}{\textbf{Input:}}
\renewcommand{\algorithmicensure}{\textbf{Output:}}
	\caption{Compressed storage method for $Q$}
 \label{alg:compressed}
	\begin{algorithmic}
		\REQUIRE discrete Fourier coefficients $\widetilde{\mathcal{A}}$, tensor indices $\bm{k}_{\mathcal{V}}$, $\bm{k}_{\mathcal{U}}$, projection matrix $\bm{P}$
        \ENSURE $Q_{ij}$
        
          \STATE
           $\bm{k}_{\mathcal{V}}\xrightarrow{\mathcal{C}} i,\quad \bm{k}_{\mathcal{U}}\xrightarrow{\mathcal{C}} j, \quad \bm{k}_{\mathcal{A}}=\bm{k}_{\mathcal{V}}\overset{N}{-}\bm{k}_{\mathcal{U}}$\\
           
          $A_{ij} = \widetilde{\mathcal{A}}_{\bm{k}_{\mathcal{A}}}$\\

          $W_{ij} = (\bm{P}\bm{k}_{\mathcal{V}})^{T}(\bm{P}\bm{k}_{\mathcal{U}})$\\

          $Q_{ij} = A_{ij} W_{ij}$
	\end{algorithmic}
\end{algorithm}

\begin{remark}\label{lemma:memory_requirement}
The memory requirement for the compressed storage method is reduced from $\mathcal{O}(D^2)$ to $\mathcal{O}(D)$.
\end{remark}

Furthermore, in preparation for the subsequent improved iterative solver, we introduce the set $\mathcal{Z}$ that comprises all non-zero elements of $Q$,
$$
 \mathcal{Z}=\left\{Q_{ij}: Q_{ij}\neq 0, ~ i,j=1,\cdots, D \right\}.
$$
We denote the compressed format of $Q$ as $\widetilde{Q}=(i, j, Q_{ij})_{Q_{ij}\in\mathcal{Z}}$. Going forward, all matrix-vector multiplications involving $Q$ will be replaced by $\widetilde{Q}$. This technique is referred to as \emph{compressed matrix-vector multiplication}.

\subsection{Preconditioner and C-PCG}
\label{sec:diag_pre_compressed_multiplication}
In this subsection, we propose an efficient  iterative approach based on conjugate gradient (CG) method to solve linear system \eqref{eqn:linear_system}. 
As \cref{tab:elliptic:two_modes:PMcond} shows, 
the linear system becomes ill-conditioned with an increase of discrete points. Therefore, it is crucial to design efficient preconditioners to reduce the condition number. Here,  we propose a diagonal preconditioner by the following matrix minimization problem \eqref{eq:precond}
\begin{align}\label{eq:precond}
M= \argmin_{\mathcal{D}\in\mathscr{D}(D)}\|Q\mathcal{D}-I_D\|_{F},
\end{align}
{where $\mathcal{D}$ is a diagonal matrix of order $D$.} The optimization problem can be solved analytically.
\begin{theorem}
The solution of \eqref{eq:precond}  is
\begin{equation}\label{eqn:preconditioner}
    M=\operatorname{diag}\big(q_{11}/\|Q \bm{e}_1\|_{2}^{2},\cdots, q_{DD}/\|Q \bm{e}_{D}\|_{2}^{2}\big).
\end{equation}
\end{theorem}
\begin{proof}
Given matrices $Q=(q_{ij})\in\mathbb{C}^{D\times D}$ and $M=\operatorname{diag}(m_{11},\cdots,m_{DD})\in\mathscr{D}(D)$, 
we can compute  $\|QM-I_D\|_F$ as follows
\begin{align*}
	\|QM-I_D\|_F^2&=	(q_{11}m_{11}-1)^2+(q_{12}m_{22})^2+\cdots+(q_{1D}m_{DD})^2\\
 &\qquad\qquad\qquad\qquad\qquad\quad\cdots\\
	&+(q_{D1}m_{11})^2+(q_{D2}m_{22})^2+\cdots+(q_{DD}m_{DD}-1)^2.
\end{align*}
The $j$-th column summation $S_j
=\big(\sum_{i=1}^{D}q_{ij}^2\big)m_{jj}^2-2q_{jj}m_{jj}+1$.
Obviously, $S_j$ is a quadratic function with respect to $m_{jj}$ and it attains its minimum if and only if $
m_{jj}= q_{jj}/\big(\sum_{i=1}^{D}q_{ij}^2 \big)= q_{jj} / \|Q 
\bm{e}_j\|_{2}^{2}.
$
\end{proof}

\begin{proposition}
    Suppose $\widetilde{\mathcal{A}}$ has $g$ non-zero elements. In this case, each row (or column) of $Q$ consists of $\mathcal{O}(g)$ non-zero elements. Therefore, the computational complexity for constructing $M$ amounts to $\mathcal{O}(gD)$.
\end{proposition}\label{lemma:preconditioner_construct}

To demonstrate the effectiveness of the diagonal preconditioner $M$, we perform a numerical test with $d=1, n=2$. The implementation of this case can be found in  \Cref{sec:numerical_experiments:elliptic:two_modes}.
As illustrated in \Cref{tab:elliptic:two_modes:PMcond}, the condition number of $Q$ exhibits a substantial increase  as $N$ grows. However, after applying the preconditioner $M$, the condition number undergoes a significant reduction and stabilizes within a very narrow range.
\begin{table}[H]
	\centering
 \footnotesize{
	\begin{tabular}{|c|c|c|c|c|c|c|}
		\hline
		$N$  &4&8&16&32&64&128\\
		\hline
 $Q$&1.77e+01&7.31e+01&2.97e+02&1.19e+03&4.81e+03&1.93e+04\\
		\hline
		$QM$&2.44&2.48&2.44&2.49&2.49&2.48\\
		\hline	
	\end{tabular}
 }
    \caption{Condition numbers of $Q$ and the preconditioned system $QM$.}
    \label{tab:elliptic:two_modes:PMcond}
\end{table}

Moreover, we use the compressed matrix-vector multiplication to accelerate the iteration process further in iterative schemes like PCG. This improved PCG method, referred to as compressed PCG (C-PCG) method, is outlined in $\Cref{alg:c-pcg}$.
\begin{algorithm}[H]
	\caption{C-PCG method}
        \label{alg:c-pcg}
	\begin{algorithmic}
		\STATE	Given initial guess vector $\bm{U}_0$ and compute the residual $\bm{r}_0=\bm{F}-\widetilde{Q}\bm{U}_0$
		
		\STATE Solve residual equation $M\bm{z}_0=\bm{r}_0,$ set $\bm{h}_0 = \bm{z}_0$
		
		\FOR{$j=0,1,\cdots,$ until convergence}
		\STATE $\alpha_j=\frac{(\bm{r}_j,\bm{z}_j)}{\widetilde{Q}\bm{h}_j,\bm{h}_j}$\\
		\STATE $\bm{U}_{j+1}=\bm{U}_j+\alpha_j\bm{h}_j$\\
		\STATE $\bm{r}_{j+1}=\bm{r}_j-\alpha_{j}\widetilde{Q}\bm{h}_j$\\
		\STATE Solve $M\bm{z}_{j+1}=\bm{r}_{j+1}$\\
		\STATE $\beta_j=\frac{(\bm{r}_{j+1},\bm{z}_{j+1})}{\bm{r}_j,\bm{z}_j}$\\
		\STATE $\bm{h}_{j+1}=\bm{z}_{j+1}+\beta_j\bm{h}_j$\\
  \ENDFOR
	\end{algorithmic}
\end{algorithm}
\begin{remark}
    Compared to standard PCG, C-PCG uses the compressed form $\widetilde{Q}$ instead of $Q$ to enhance matrix-vector multiplication, reducing the computational complexity from $\mathcal{O}(D\times D)$ to $\mathcal{O}(gD)$, equivalent to the costs of compressed matrix-vector multiplication. Moreover, it is worth noting  that we can substitute the PCG method with other iterative approaches if desired.
\end{remark}


\section{Numerical experiments}
\label{sec:numerical_experiments}

In this section, we present two classes of numerical experiments to demonstrate the performance of our algorithm. Firstly, we  employ the PM to solve the quasiperiodic elliptic equation \eqref{eqn:elliptic} and compare its efficiency with the PAM. We also compare the performance of PCG and C-PCG in solving the linear system discretized by PM. Consequently, we focus on the quasiperiodic homogenization problem and demonstrate the accuracy and efficiency of our algorithm in calculating the homogenized coefficients. These experiments  are conducted using MATLAB R2017b on a laptop computer equipped with an Intel Core 2.50GHz CPU and 4GB RAM. The computational time, measured in seconds(s), is referred to as CPU time.

\subsection{Quasiperiodic elliptic equations}
\label{sec:numerical_experiments:elliptic}

In this subsection,  we present two examples to demonstrate the effectiveness of the proposed algorithm. In the first example, the spectral points of the elliptic coefficient and the solution are two separated incommensurate frequencies, namely $1$ and $\sqrt{2}$. We begin with implementing the PM to solve the quasiperiodic elliptic equation \eqref{eqn:elliptic} using the C-PCG method. Our numerical experiments involve comparing the CPU time and memory usage between standard PCG and C-PCG  to demonstrate the high efficiency of C-PCG. Furthermore, we give a comparative analysis of the accuracy and efficiency of PM against the PAM, emphasizing the advantages of PM in solving quasiperiodic elliptic equations, particularly in avoiding the Diophantine approximation error. The second example is more complex,
its spectral point set of solution consists of linear combinations of $1$ and $\sqrt{2}$. More examples can be found in \cite[pp.~20--23]{jiang2024projection}.

In what follows, we denote the numerical solutions by $u_N$ and measure the numerical error in the $\mathcal{L}^2_{QP}(\mathbb{R}^d)$-norm as $e_N = \|u_N - u\|_{\mathcal{L}^2_{QP}(\mathbb{R}^d)}$, where $u$ is the exact solution of \eqref{eqn:elliptic}. The order of convergence is calculated by $\kappa=\ln(e_{N_1}/e_{N_2})/ \ln (N_1/ N_2)$.


\subsubsection{Two separated incommensurate frequencies}
\label{sec:numerical_experiments:elliptic:two_modes}

In the first test, we aim to demonstrate the accuracy and efficiency of the PM together with C-PCG iterative solver. For the quasiperiodic elliptic equation \eqref{eqn:elliptic}, we choose the coefficient $\alpha_1(x) = \cos(2\pi x) + \cos(2\pi\sqrt{2}x) + 6$ and the exact solution $u_1(x) = \sin(2\pi x) + \sin(2\pi\sqrt{2}x)$. 

\textbf{C-PCG \textit{vs.} PCG in PM scheme.} Within the PM framework, we compare the efficiency, memory usage, and condition number of  PCG and C-PCG. 
\Cref{tab:elliptic:two_modes:PMcond} has presented the condition number of $Q$ and $QM$, and shown that our proposed diagonal preconditioner $M$ is very efficient. 
\Cref{tab:elliptic:two_modes:PMefficiency} compares the efficiency of PCG with C-PCG. We can find that the C-PCG can significantly reduce the CPU time compared with the PCG. As $N$ grows, the reduction in CPU time diminishes rapidly. Meanwhile, the C-PCG achieves significant memory savings, as shown in \Cref{tab:elliptic:two_modes:PMmemory}. We denote the memory usage (in units of Gb) of PCG and C-PCG as $M_{\mathrm{PCG}}$ and $M_{\mathrm{C-PCG}}$, respectively, with the memory ratio $r = M_{\mathrm{PCG}}/M_{\mathrm{C-PCG}}$ is approximately $\mathcal{O}(D)$. 


\begin{table}[!hbpt]
	\centering
 \footnotesize{
	\begin{tabular}{|c|c|c|c|c|c|c|}
		\hline
		$N$  &4&8&16&32&64&128\\
		\hline
		 CPU time(PCG)&0.05&0.78&9.63&135.60&-&-\\

		\hline
		CPU time(C-PCG)&0.02&0.07&0.27&1.04&4.89&17.32\\
		\hline	
		Iteration(PCG) & 13& 19&19&19&-&-\\
		\hline
		Iteration(C-PCG) & 13& 19&19&19&19&19\\
		\hline
	  $ e_N$(PCG)&6.24e-02&3.02e-16&5.09e-16&5.66e-16&-&-\\
        \hline
		$ e_N$(C-PCG)&6.24e-02&3.02e-16&5.10e-16&5.66e-16&4.35e-16&3.62e-16\\
       \hline

	\end{tabular}
 }
   \caption{
Efficiency comparison of PCG  and C-PCG  when solving \eqref{eqn:elliptic} with ($\alpha_1(x), u_1(x)$) in PM (Data for $N\geq 64$ with  PCG is not available due to insufficient memory).}\label{tab:elliptic:two_modes:PMefficiency}
\end{table}

\begin{table}[!hbpt]
	\centering
 \footnotesize{

	\begin{tabular}{|c|c|c|c|c|c|c|}
		\hline
		$N$  &4&8&16&32&64&128\\
		\hline
		 $M_{\mathrm{PCG}}$ &1.00e-03&1.60e-02&2.50e-01&4.00e+00&6.40e+01&1.02e+03\\

		\hline
		$M_{\mathrm{C-PCG}}$ &6.11e-05&2.44e-04&9.77e-04&3.90e-03&1.56e-02&6.25e-02\\
		\hline	
		$r$ &1.60e+01&6.40e+01&2.56e+02&1.02e+03&4.10e+03&1.64e+05\\
        \hline
	\end{tabular}
 }
    \caption{{Comparison of memory usage between full storage ($M_{\mathrm{PCG}}$) and compressed storage ($M_{\mathrm{C-PCG}}$) for the stiffness matrix $Q$ generated by $\alpha_1$. }}\label{tab:elliptic:two_modes:PMmemory}
 \end{table}

\textbf{PM \textit{vs.} PAM.} We provide a comprehensive comparison between the PM and the PAM in solving \eqref{eqn:elliptic}. The implementation of PAM can refer to \cite{jiang2014numerical, jiang2023approximation}. We use periodic functions $\alpha_{1p}(x) = \cos(2\pi x) + \cos(2\pi([\sqrt{2}L]/L) x) + 6$ with varying $L$, representing the length of the computational area, to approximate the quasiperiodic coefficient $\alpha_1(x) = \cos(2\pi x) + \cos(2\pi\sqrt{2} x) + 6$. In PAM, we choose $E = L\times N$ discrete points to ensure enough numerical accuracy of discretizing \eqref{eqn:elliptic} with $N=16$. For convenience, we continue to use $e_N$ to denote the numerical error of PAM. We only present the results when
$L = 2, 5, 12, 29, 70, 169, 408 $ for PAM. \Cref{fig:elliptic:dioerr} shows the Diophantine approximation error $|L\sqrt{2}-[L\sqrt{2}]|$ with respect to $L$, intending to explain why these particular $L$ are chosen for approximation.
Our tests not only involve computing the numerical error $e_N$, but also giving a qualitative relationship between the numerical error and the Diophantine approximation error, as presented in \Cref{tab:elliptic:two_modes:PAMerr}. As we can see, $e_N$  is mainly controlled by the Diophantine approximation error. 
Furthermore, we  fix $L$ and gradually increase {$N= 2^l, l \in \mathbb{N}^+$}.  In  \Cref{tab:elliptic:two_modes:PAMvaryN}, we observe that PAM fails to achieve improved accuracy even with a significantly higher number of discrete points. The numerical error stagnates once $N$ exceeds a certain value. 
More theoretical results on function approximation theory in the context of PAM can refer to \cite{jiang2023approximation}.

\begin{table}[!hbpt]
	\centering
 \footnotesize{
	\begin{tabular}{|c|c|c|c|c|c|c|c|}
		\hline
        $L$&$2$&$5$&$12$&$29$&$ 70$&$169$&$ 408$\\
		\hline
    $[\sqrt{2}L]/L$&3/2&7/5&17/12&41/29&99/70&239/169&577/408\\
		\hline
		
		$|L\sqrt{2}-[L\sqrt{2}]|$&1.72e-01&7.11e-02&2.94e-02&1.22e-02&5.10e-03&2.10e-03&8.67e-04\\
		\hline
		$ e_N$&2.12e-01&9.18e-02&3.77e-02&1.56e-02&6.50e-03&2.70e-03&1.10e-03\\
		\hline
	\end{tabular}
	}
   \caption{ Numerical error of PAM in solving \eqref{eqn:elliptic} with $\alpha_{1p}(x)$ and corresponding Diophantine approximation  error against $L$.   }\label{tab:elliptic:two_modes:PAMerr}
   \end{table}

   \begin{figure}[!hbpt]
     \centering
     \includegraphics[width=10cm]{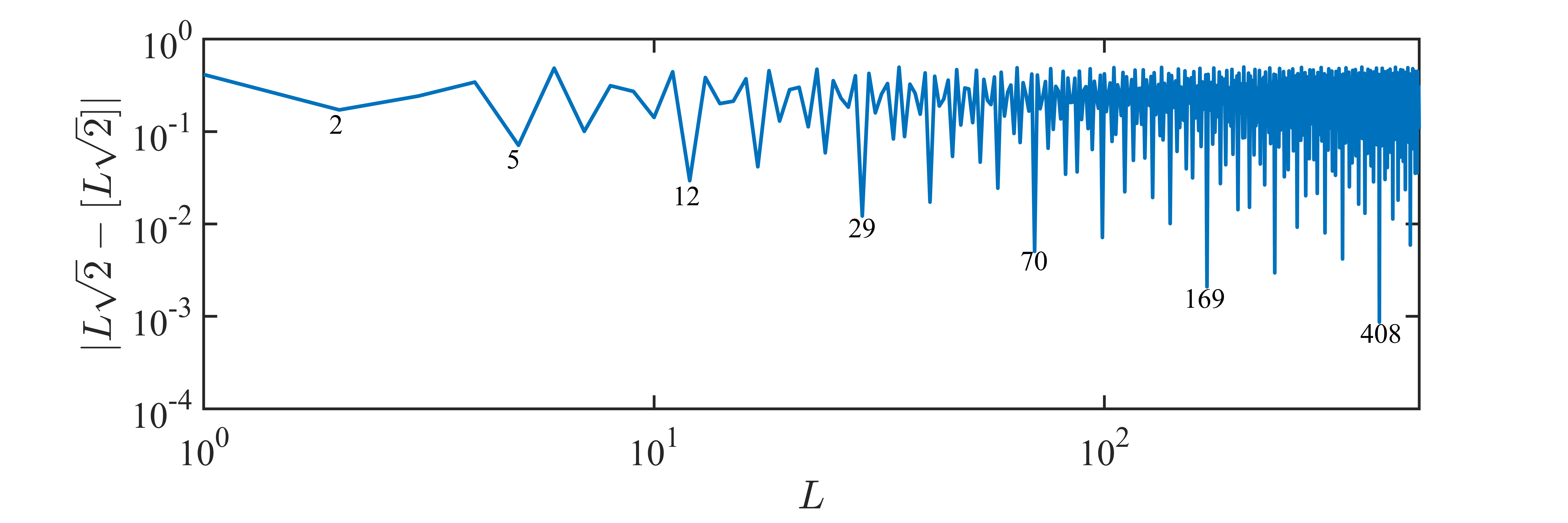}
     \caption{Diophantine approximation error of $\sqrt{2}$. }
     \label{fig:elliptic:dioerr}
 \end{figure}

 \begin{table}[!hbpt]
	\centering
 \footnotesize{
	\begin{tabular}{|c|c|c|c|c|c|}
		\hline
		$e_N$&$l=2$&$l=3$&$l=4$&$l=5$ &$l=6$\\
		\hline
		 $ L=2$&5.55e-01&2.17e-01&2.12e-01&2.11e-01&2.11e-01\\
		 \hline
		 $L=5$&5.41e-01&9.17e-02&9.18e-02&9.18e-02&9.18e-02\\
		 \hline
		 $L=12$&5.27e-01&3.77e-02&3.77e-02&3.77e-02&3.77e-02\\
		 \hline
		 $L=29$&5.07e-01&1.56e-02&1.56e-02&1.56e-02&1.56e-02\\
		 \hline
	\end{tabular}
 }
    \caption{Numerical error with variable $E=L\times 2^l$ for different $L$ in PAM.  }\label{tab:elliptic:two_modes:PAMvaryN}
\end{table}

We then compare the efficiency of PM and PAM. \Cref{tab:elliptic:two_modes:PMPAMcompare} presents the accuracy achieved by these two methods, along with their required CPU time.  It is evident that PM avoids the impact of Diophantine approximation errors and achieves machine precision in a short amount of time. In contrast, PAM, even with much larger CPU time, only achieves relatively low accuracy, which is dependent on the Diophantine approximation error.

\begin{table}[!hbpt]
	\centering
 \footnotesize{
	\begin{tabular}{|c|c|c|c|c|}
		\hline
		&$N$&4&8&16\\
		\hline
		\multirow{4}{*}{$e_N$} & \mbox{PM} & 6.24e-02 & 3.02e-16 & 5.10e-16\\
		\cline{2-5}
		&\mbox{PAM($L=70$)}& 5.02e-01 & 6.50e-03 & 6.50e-03\\
		&\mbox{PAM($L=169$)}& 4.98e-01 & 2.70e-03 & 2.70e-03\\
		&\mbox{PAM($L=408$)}& 3.97e-01 & 1.10e-03 & 1.10e-03\\
		\hline
		\multirow{4}{*}{\mbox{CPU time}}&\mbox{PM} & 1.77e-02 & 7.08e-02 & 1.73e-01\\
		\cline{2-5} 
		& \mbox{PAM($L=70$)} &7.23e-02 &2.53e-01& 6.52e-01\\
		& \mbox{PAM($L=169$)} &1.42e-01 &5.78e-01& 1.53\\

		& \mbox{PAM($L=408$)} &2.56e-01 &1.60& 5.30\\
		\hline
		
	\end{tabular}
 }
  \caption{{Efficiency comparison between PM and PAM with different periods $L=70, 169, 408$  in solving  \eqref{eqn:elliptic} with $\alpha_1(x)$. }}\label{tab:elliptic:two_modes:PMPAMcompare}
 \end{table}


 
\subsubsection{Combination of two incommensurate frequencies}
\label{sec:numerical_experiments:elliptic:linear_combination}

To demonstrate the {polynomial accuracy} of the PM  and highlight the high efficiency of C-PCG further, we design a more complex example where the spectral points of solution $u$ are formed by a linear combination of 1 and $\sqrt{2}$. We adopt the  coefficient $\alpha_1(x)$ and the exact solution $u_2(x)$ in the following form
$$
u_2(x)=\sum_{ \lambda\in\Lambda_K}\widetilde{u}_\lambda e^{\imath 2\pi \lambda x},\quad \widetilde{u}_\lambda=e^{-(|k_1|+|k_2|)},\quad \Lambda_K =\left\{\lambda=k_1 +k_2\sqrt{2}:-32\leq k_1, k_2< 32\right\}. 
$$

Similar to the first test, we evaluate the performance of PCG and C-PCG by comparing  their numerical error and CPU time.  \Cref{tab:elliptic:linear_combination:PMefficiency} once again demonstrates that C-PCG solves \eqref{eqn:elliptic} more efficiently. Furthermore, thanks to the memory savings of C-PCG, we observe the exponential decay of $e_N$ to machine precision, sufficiently covering all non-zero values of 
$\widetilde{u}_{\lambda}$. Additionally, the error order $\kappa$  increases with $N$.

As illustrated in \Cref{fig:elliptic:covergence_plot}, the error consistently decreases as $N$ doubles. In this log-log convergence plot, the magnitude of the slope increases steadily, indicating an accelerating rate of error decay. This behavior suggests that the error order $\kappa$ improves with increasing $N$. Once $N$ becomes sufficiently large to capture all significant non-zero values of $\widetilde{u}_{\lambda}$, the PM achieves machine precision, as evidenced by the flattening of the error curve at the end. Since all functions involved in the test are analytic (i.e., $s = +\infty$ in \Cref{thm:convergence_analysis}), the PM exhibits spectral accuracy in this example. These results are consistent with, and further support, our theoretical convergence analysis.

\begin{table}[!hbpt]
	\centering
 \footnotesize{
	\begin{tabular}{|c|c|c|c|c|c|c|}
		\hline
		$N$  &4&8&16&32&64&128\\
			\hline
		 $ e_N$ (PCG)&2.73&2.02e-01&5.72e-04&8.32e-08&-&-\\
		
		\hline
		 $ e_N$(C-PCG)&2.73&2.02e-01&5.72e-04&8.32e-08&2.42e-15&2.46e-15\\
   \hline
        $\kappa$&-&3.75 & 8.46& 12.75&18.75&-\\
		 \hline
		 CPU time(PCG)&2.80e-03&4.26e-02&7.16e-01&1.13e+01&-&-\\
		 \hline
		 CPU time(C-PCG)&7.08e-04&3.01e-03&1.12e-02&4.84e-02&1.57e-01&5.37e-01\\
		 \hline
		 Iteration (PCG)& 14&19&18&18&-&-\\
		 \hline
		 Iteration (C-PCG)& 14&19&18&18&18&18\\
		\hline
	\end{tabular}
 	\caption{ 
Efficiency comparison of PCG   and C-PCG  when solving \eqref{eqn:elliptic} with ($\alpha_1(x), u_2(x)$) in PM (Data for $N\geq 64$ with  PCG  is not available due to insufficient memory).}\label{tab:elliptic:linear_combination:PMefficiency}}
\end{table}

\begin{figure}[!hbpt]
      \centering
      \includegraphics[width=7cm]{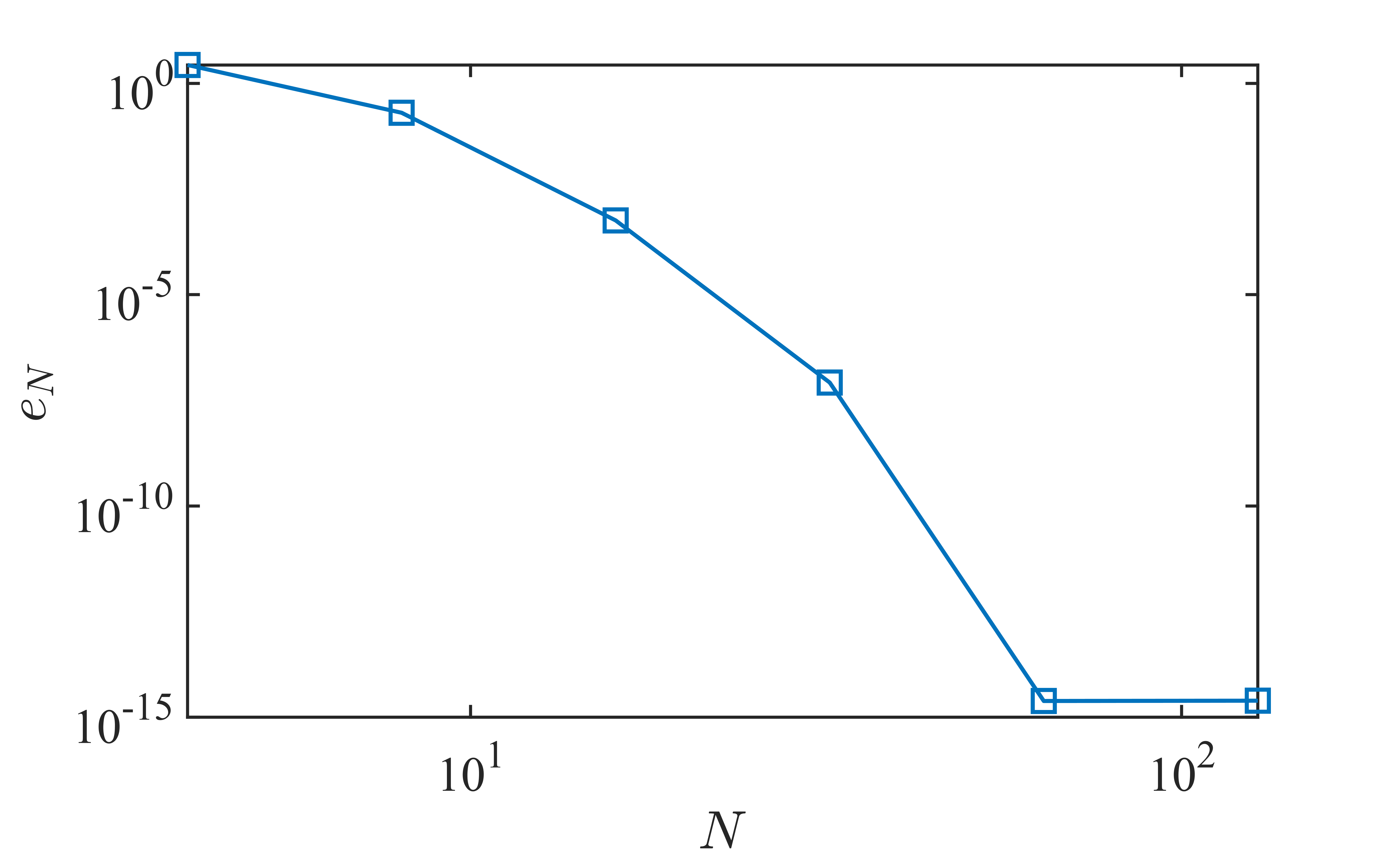}
      \caption{{The error  $e_N$ against $N$ with $(\alpha_1, u_2)$ in PM, plotted on a log-log scale.} }
      \label{fig:elliptic:covergence_plot}
  \end{figure}

\subsubsection{Three separated incommensurate frequencies}\label{sec:numerical_three_modes}
In this subsubsection, we present the test where the spectral points of the elliptic coefficient and the solution correspond to three separated incommensurate frequencies. Similar to the first test, we continue using the PM to solve the quasiperiodic elliptic equation \eqref{eqn:elliptic} with the C-PCG method, demonstrating the effectiveness of our proposed algorithm.

In this test, we consider the quasiperiodic elliptic equation  with  the coefficient $$\alpha_2(x)=\cos(2\pi x)+\cos(2\pi\sqrt{2}x)+\cos(2\pi\sqrt{3}x)+6$$ and the exact solution $$u_3(x)=\sin(2\pi x)+\sin(2\pi\sqrt{2}x)+\sin(2\pi\sqrt{3}x).$$

\textbf{C-PCG \textit{vs.}~PCG in PM scheme.}  Firstly, we apply the PM with the C-PCG method to solve the quasiperiodic elliptic equation \eqref{eqn:elliptic} and compare its efficiency with the standard PCG method. 
\Cref{tab_appendix:elliptic:three_modes:PMefficiency} shows a direct comparison of the efficiency between PCG and C-PCG. It is evident that C-PCG can significantly reduce CPU time compared to PCG.  At the same time, C-PCG achieves substantial memory savings, as illustrated in \Cref{tab_appendix:elliptic:three_modes:PMmemory}.  Visually, as shown in \Cref{fig_appendix:elliptic:three_modes:PMmemory}, the memory ratio $r = M_{\mathrm{PCG}}/M_{\mathrm{C-PCG}}$ scales approximately $\mathcal{O}(D)$. For example, when $N=32$,  C-PCG  requires only 1.25e-01 Gb of memory, whereas PCG requires 4.10e+03 Gb.  This results in the memory  consumed by the two reaching  $32800 \approx 32^3$.    Moreover, \Cref{tab_appendix:elliptic:three_modes:PMcond} presents the performance of preconditioner $M$. After applying 
$M$, the condition number of $Q$ undergoes a substantial reduction, stabilizing in the range of 3 to 3.5 as $N$ grows.
 \begin{table}[!hbpt]
	\centering
 \footnotesize{
	\begin{tabular}{|c|c|c|c|c|c|}
		\hline
		$N$  &4& 8 & 16&32&64\\
		\hline
		 CPU time(PCG)&4.31e-02&2.76&-&-&-\\
		\hline
		CPU time(C-PCG)&4.10e-03 &3.12e-02&2.51e-01& 2.05&15.87\\
		\hline
		Iteration(PCG)&20&24&-&- &-\\
		\hline
	  Iteration(C-PCG)&20&24&24&23 &23\\
		\hline
		$ e_N$(PCG)&7.79e-02&6.16e-16&-&-&-  \\
		\hline
		$ e_N$(C-PCG)&7.79e-02&6.16e-16&8.10e-16&1.88e-15&1.86e-15\\
		\hline
	\end{tabular}
 }
	 \caption{ 
Efficiency comparison of PCG and C-PCG when solving \eqref{eqn:elliptic} with ($\alpha_2(x), u_3(x)$) in PM (Data for $N\geq 16$ with  PCG is not available due to insufficient memory).}\label{tab_appendix:elliptic:three_modes:PMefficiency}
\end{table}

\begin{figure}[!hbpt]
     \centering
     \includegraphics[width=7cm]{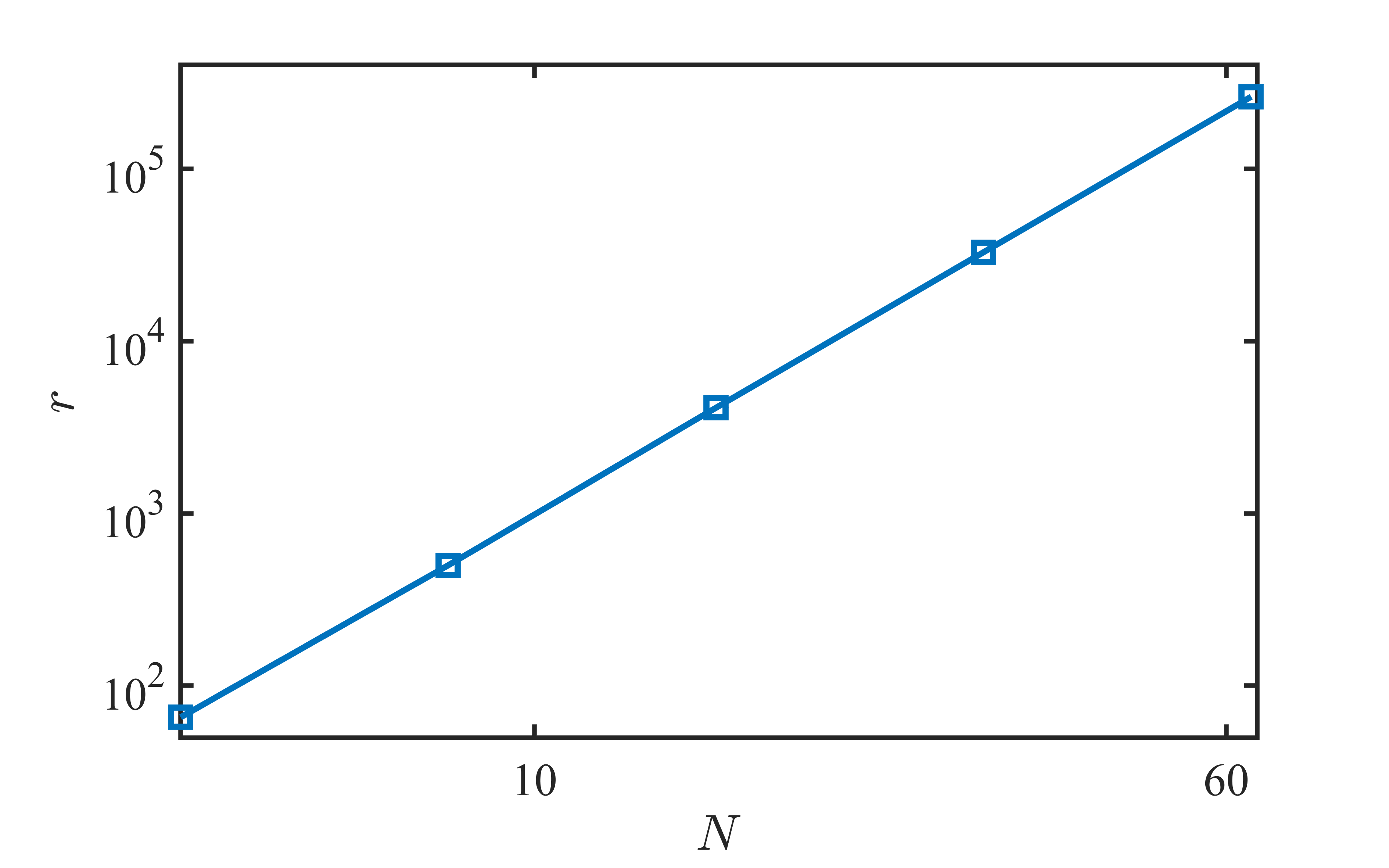}
     \caption{The ratio  $r=M_{\mathrm{PCG}}/M_\mathrm{C-PCG}$ vs.   $N$ in log-log scale. }
     \label{fig_appendix:elliptic:three_modes:PMmemory}
 \end{figure}

\begin{table}[!hbpt]
	\centering
 \footnotesize{
	\begin{tabular}{|c|c|c|c|c|c|}
		\hline
		$N$ &4&8&16&32&64\\
		\hline
		 $M_{\mathrm{PCG}}$&1.60e-02&1.00e+00&6.40e+01&4.10e+03&2.62e+05\\

		\hline
		$M_{\mathrm{C-PCG}}$&2.44e-04&2.00e-03&1.56e-02&1.25e-01&1.00e+00\\
		\hline

	\end{tabular}
 }
    \caption{Comparison of memory usage between full storage ($M_{\mathrm{PCG}}$) and compressed storage ($M_{\mathrm{C-PCG}}$) for the stiffness matrix $Q$ generated by $\alpha_2$.  }
    \label{tab_appendix:elliptic:three_modes:PMmemory}
\end{table}

\begin{table}[!hbpt]
	\centering
 \footnotesize{
	\begin{tabular}{|c|c|c|c|c|c|}
		\hline
		$N$  &4&8&16&32&64\\
		\hline
		 $Q$&3.88e+01&1.56e+02&6.25e+02&2.50e+03&5.61e+04\\

		\hline
		$QM$&3.27&3.27&3.13&3.07&3.09\\
		\hline	
	\end{tabular}
 }
 \caption{Condition number of $Q$ and preconditioned system $QM$.}
  \label{tab_appendix:elliptic:three_modes:PMcond}
\end{table}

\textbf{PM \textit{vs.}~PAM.}
We present a detailed comparison of the PM and PAM in solving equation \eqref{eqn:elliptic}. To approximate $\alpha_2(x)$, we use periodic functions
$
\alpha_{2p}(x)=\cos(2\pi x)+\cos(2\pi([\sqrt{2}L]/L)x)+\cos(2\pi([\sqrt{3}L]/L) x)+6
$ with varying $L$. The corresponding Diophantine approximation error, denoted as $e_d$ is defined by 
$$
e_d=\max\left\{ \sqrt{2}L-[\sqrt{2}L], \sqrt{3}L-[\sqrt{3}L]\right\}.
$$
To ensure enough numerical accuracy of discretizing \eqref{eqn:elliptic}, we take $E = 16L$  in the PAM. 
\Cref{tab_appendix:elliptic:three_modes:PAMerr} records the numerical error $e_N$ of PAM and corresponding Diophantine approximation error $e_d$. The data shows that $e_N$ of PAM is mainly controlled by $e_d$.
Moreover, once the value of $L$ is fixed, the discrete points reach a critical value,  beyond which the numerical error $e_N$ of PAM cannot decrease,  as illustrated in \Cref{tab_appendix:elliptic:three_modes:PAMvaryN}. For instance, when $L=34$,  despite the increase in  $k$, the numerical error of PAM remains at $e_N=1.17\mathrm{e}-01$, which is  comparable to the order of the Diophatine approximation error.
\begin{table}[!hbpt]
	\centering
 \footnotesize{
	\begin{tabular}{|c|c|c|c|c|c|}
        \hline
		$N$&$3\times16$&$7\times 16$&$22\times 16$&$34\times 16$&$ 41 \times 16$\\
		\hline
		$[\sqrt{2}L]/L$&4/3&10/7&31/22&48/34&58/41\\
		\hline
		$[\sqrt{3}L]/L$&5/3&12/7&38/22&59/34&71/41\\
		\hline
		 $e_d$&2.43e-01&1.24e-01&1.13e-01&1.10e-01&1.72e-02\\
		\hline
		$ e_N$&2.63e-01&1.75e-01&1.96e-01&1.17e-01&2.87e-02\\
		\hline
		
	\end{tabular}
 }
 \caption{ Numerical error of PAM in solving \eqref{eqn:elliptic} with $\alpha_{2p}(x)$ and corresponding Diophantine approximation  error against $L$.  }
	\label{tab_appendix:elliptic:three_modes:PAMerr}
\end{table}

\begin{table}[!hbpt]
	\centering
 \footnotesize{
	\begin{tabular}{|c|c|c|c|c|}
		\hline
	  $e_N$&$k=3$&$k=4$&$k=5$ &$k=6$\\
		\hline
		 $L=3$&2.62e-01&2.62e-01&2.63e-01&2.63e-01\\
		\hline
		$L=7$&1.75e-01&1.75e-01&1.75e-01&1.75e-01\\
		\hline
		$L=22$&1.96e-01&1.96e-01&1.96e-01&1.96e-01\\
		\hline
		$L=34$&1.17e-01&1.17e-01&1.17e-01&1.17e-01\\
		\hline
		$L=41$&2.87e-02&2.87e-02&2.87e-02&2.87e-02\\
		\hline
	\end{tabular}
 }
 \caption{Numerical error with variable $E=L\times 2^k$ for different $L$ in PAM. }
 \label{tab_appendix:elliptic:three_modes:PAMvaryN}
\end{table}

Finally, we  compare the efficiency of PM and PAM. \Cref{tab_appendix:elliptic:three_modes:PMPAMcompare} presents the accuracy achieved by PM and PAM, along with their required CPU time. Moreover, \Cref{fig_appendix:elliptic:three_modes:PMPAMcompare} provides a visual representation of the trade-off between cost  and accuracy. It again showcases  that PM achieves machine precision in a short amount of time, while PAM only achieves relatively low accuracy dependent on the Diophantine approximation error, even with much larger CPU time.

\begin{table}[!hbpt]
	\centering
 \footnotesize{
	\begin{tabular}{|c|c|c|c|c|}
		\hline
		&$N$&4&8&16\\
		\hline
		\multirow{3}{*}{$e_N$} & PM & 7.79e-02 & 6.16e-16 & 8.10e-16\\
		\cline{2-5}
		&\mbox{PAM($L=1183$)}& 2.56e-01 &2.79e-03 &2.79e-03\\
		&\mbox{PAM($L=1463$)}& 2.55e-01 &2.25e-03 & 2.25e-03\\

		\hline
		\multirow{3}{*}{\mbox{CPU time}}&\mbox{PM} & 3.99e-02 & 2.59e-02 & 2.51e-01\\
		\cline{2-5} 
		& \mbox{PAM($L=1183$)} &4.56 &17.26 &70.23 \\
		& \mbox{PAM($L=1463$)} &5.43 &21.34& 81.63\\
		\hline
		
	\end{tabular}
      }
      \caption{Performance comparison between PM and PAM in solving \eqref{eqn:elliptic} with $\alpha_2(x)$. }\label{tab:elliptic:two_modes:PMPAMcompare }
	\label{tab_appendix:elliptic:three_modes:PMPAMcompare}
\end{table}

\begin{figure}[!hbpt]
     \centering
     \includegraphics[width=7cm]{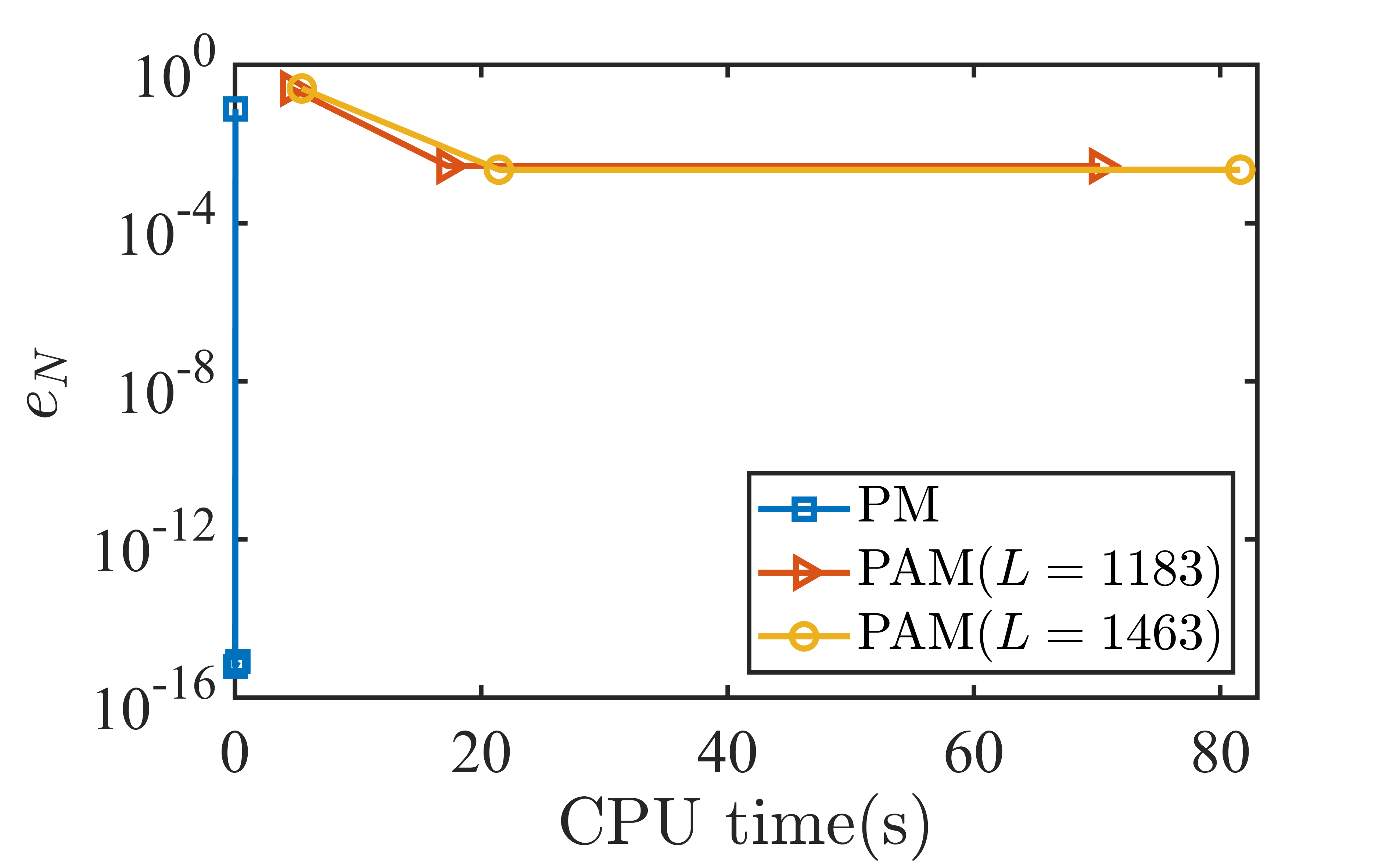}
     \caption{Cost-Accuracy trade-off  of  PM  and PAM $(L=1183, 1463)$ in solving \eqref{eqn:elliptic} with $\alpha_2(x)$,  respectively.}
     \label{fig_appendix:elliptic:three_modes:PMPAMcompare}
 \end{figure}

\textbf{Different frequencies in coefficient and exact solution.}
In this test, we  consider the case where the spectral points of coefficients and solutions are distinct. We discretize \eqref{eqn:elliptic} using the PM  with $\alpha_1(x) = \cos(2\pi x) + \cos(2\pi\sqrt{2} x) + 6$ and the exact solution $u_4(x) = \sin(2\pi x) + \sin(2\pi \sqrt{3}x)$. The resulting linear system is solved using the C-PCG method. It is apparent that this quasiperiodic system can be embedded into a three-dimensional parent system. The numerical results demonstrate that the numerical error $e_N$ of the PM  convergence rapidly to machine precision, as shown in \Cref{tab_appendix:elliptic:incommensurate}.
\begin{table}[!hbpt]
	\centering
 \footnotesize{
	\begin{tabular}{|c|c|c|c|c|c|}
		\hline
		$N$  &4& 8 & 16&32&64\\
		\hline
		CPU time &3.70e-03 &2.51e-02&2.47e-01& 1.10&8.69\\
		\hline
	   	Iteration &20&24&24&23 &23\\
		\hline
		$ e_N$ &7.10e-02&6.26e-16&7.95e-16&1.79e-15&1.74e-15\\
		\hline
	\end{tabular}
 }
 \caption{Numerical solution of \eqref{eqn:elliptic} with $(\alpha_1(x), u_4(x))$ using C-PCG method in PM.}
	 \label{tab_appendix:elliptic:incommensurate}
\end{table}

\subsubsection{Four separated incommensurate frequencies}\label{sec:numerical_four_modes}
In this subsubsection, we further demonstrate the high precision and efficiency of the proposed method for solving \eqref{eqn:elliptic} through a two-dimensional example. Specifically, we consider the case where $$\alpha_3(x,y) = \cos(2\pi x) + \cos(2\pi\sqrt{2}y) + \cos(2\pi\sqrt{3}x) + \sin(2\pi\sqrt{5}y) + 12$$ and $$u_5(x,y) = \sin(2\pi x) + \cos(2\pi\sqrt{2}y) + \sin(2\pi\sqrt{3}x) + \cos(2\pi\sqrt{5}y)$$ in  \eqref{eqn:elliptic}. Using the  projection matrix
$$
\bm{P}=2\pi\cdot\left(
\begin{array}{cccc}
	1 & 0& \sqrt{3} &0\\
    0 &\sqrt{2}& 0&\sqrt{5}\\
\end{array}
\right),
$$
the quasiperiodic system can be embedded into a four-dimensional parent system.

\Cref{tab_appendix:elliptic:four_modes:PMerr} shows that $e_N$ rapidly reaches machine precision in a very short amount of time.  The rapid convergence highlights exceptional efficiency and accuracy of PM, further validating its effectiveness in solving the quasiperiodic elliptic problem.

\begin{table}[!hbpt]
    \centering
    \footnotesize{
    \begin{tabular}{|c|c|c|c|}
    \hline
        $N$ &4 &6&8 \\
        \hline
        $ e_N$ &4.25e-02&2.41e-16 &5.39e-16\\
        \hline
        CPU time&0.31&1.57&4.97\\
        \hline
    \end{tabular}
      }
\caption{ Numerical solution with $(\alpha_3(x, y), u_5(x, y))$ using PM with C-PCG method.}
\label{tab_appendix:elliptic:four_modes:PMerr}
\end{table}

\subsection{Quasiperiodic homogenization}
\label{sec:numerical_experiments:hom}

In this subsection, we apply PM to quasiperiodic homogenization. In classical two-scale homogenization, we consider the coefficient $A(\bm{x}/\varepsilon)$, where $\varepsilon \ll 1$. The corresponding solution is denoted as $u_\varepsilon$, and as $\varepsilon \to 0$, we expect the convergence of $u_\varepsilon$ to $u^*$, which represents the solution of an elliptic equation with a homogenized coefficient $A^*$. Alternatively, we can perform a scaling transformation of the variable $\bm{x}/\varepsilon \to \bm{x}$, which allows us to transform the domain from a bounded region to the entire space. When $A(\cdot)$ is periodic, it is possible to construct the corrector equation on the periodic cell. However, when $A(\cdot)$ is quasiperiodic, the corresponding quasiperiodic corrector equation is posed on the whole space 
\begin{align}\label{eqn:corrector}
    -\mbox{div}(A_{\mathrm{q-per}}(\bm{x})(\bm{p}+\nabla u_{\bm{p}}(\bm{x}))) = 0, \quad\bm{x}\in\mathbb{R}^d,
\end{align}
where the scalar-valued coefficient matrix $A_{\mathrm{q-per}}(\bm{x})$ is uniformly elliptic, and $\bm{p}\in\mathbb{R}^{d}$ is an arbitrary vector. According to the homogenization theory \cite{blanc2023homogenization}, $A^*$ can be computed by the following formula
\begin{equation}\label{eqn:homcoefficient}
    A^*_{ij}=\mathcal{M}\left\{ \bm{e}_i^{T}A_{\mathrm{q-per}}(\bm{x})(\bm{e}_j+\nabla u_{\bm{e}_j}(\bm{x})) \right\}.
\end{equation}
The numerical challenge lies in accurately solving the quasiperiodic corrector equation \eqref{eqn:corrector}.
Our developed algorithm addresses this crucial issue by leveraging the periodic structure in the high dimensions, thereby enhances the overall efficiency and accuracy. To demonstrate the effectiveness of our approach, we provide a numerical example in the following. 


We consider the corrector equation \eqref{eqn:corrector} with 
\begin{align}{\label{eqn:homcoefficient1}}
A_{\mathrm{q-per}}(x_1, x_2)=
\begin{pmatrix}
      \begin{smallmatrix}
     &4+\cos(2\pi(x_1+x_2))+ \cos(2\pi\sqrt{2}(x_1+x_2))
  \end{smallmatrix}
	 &0\\
	0& \begin{smallmatrix}
	     & 6+\sin^2(2\pi x_1)+ \sin^2(2\pi\sqrt{2}x_1)
	\end{smallmatrix}
\end{pmatrix},
\end{align}
which also appears in \cite{blanc2010improving}.
The corresponding projection matrix is 
$$
\bm{P}=2\pi\cdot\left(
\begin{array}{cccc}
	1 & \sqrt{2}& 1 &\sqrt{2}\\
    1 &\sqrt{2}& 0&0\\
\end{array}
\right).
$$

We calculate a reference approximation of $A^*$ using PM discretization with $N=18$, denoted as $A_{\mathrm{ref}}^*$. Then, we compute the numerical homogenized coefficients $A^*_{11, N}$ and $A^*_{22, N}$, and record the corresponding numerical errors $e^*_{11, N}=|A^*_{11, N}-A^*_{11, \mathrm{ref}}|$ and $e^*_{22, N}=|A^*_{22, N}-A^*_{22, \mathrm{ref}}|$ in \Cref{tab:hom:four_modes:PMa11}. These results demonstrate that the errors $e^*_{11, N}$ and $e^*_{22, N}$
 {decay at a polynomial rate} as 
$N$ increases, with $A_{\mathrm{ref}}^*$ already possessing a minimum accuracy of 10 digits. 

\begin{table}[!hbpt]
    \centering
    \footnotesize{
    \begin{tabular}{|c|c|c|c|c|c|c|c|}
    \hline
    $N$ &4 & 6&8 &10&12&14&16 \\
    \hline
    $ e_{11, N}^*$  &1.60e-03&3.18e-05&7.23e-07&1.77e-08&4.56e-10&1.21e-11&3.17e-13\\
    \hline
    $e_{22, N}^*$ &8.11e-03&1.87e-04&4.37e-06&1.03e-07&2.45e-09&5.89e-11&1.45e-12\\
    \hline
    
    \end{tabular}
    }
    \caption{Numerical errors for homogenized coefficients $A^*_{11}$ and $A^*_{22}$ calculated by the C-PCG method with PM discretization.}
    \label{tab:hom:four_modes:PMa11}
\end{table}

{For comparison, we also compute $A_{11}^*$ by the PAM and the FM. The implementation of the FM \cite{blanc2010improving, blanc2023homogenization} is provided in \Cref{sec_appendix:solve_corrector_eqn_filter}. In the PAM, we take
\begin{align} {\label{eqn:homcoefficientperiodic}}
A_{\mathrm{per}}(x_1, x_2)=
\begin{pmatrix}
      \begin{smallmatrix}
     &4+\cos(2\pi(x_1+x_2))+\cos(2\pi([\sqrt{2}L]/L)(x_1+x_2))
  \end{smallmatrix}
	 &0\\
	0& \begin{smallmatrix}
	     & 6+\sin^2(2\pi x_1)+ \sin^2(2\pi([\sqrt{2}L]/L)x_1)
	\end{smallmatrix}
\end{pmatrix}
\end{align}  
to approximate \eqref{eqn:homcoefficient1} by varying computational area $[0,L)^2$. The homogenized coefficients in the periodic case reads
\begin{equation}
    A^*_{ij}=\frac{1}{L^2}\int_{[0, L)^2}\bm{e}_i^{T}A_{\mathrm{per}}(\bm{x})(\bm{e}_j+\nabla u_{\bm{e}_j}(\bm{x}))\,d\bm{x}.
\end{equation} 
For the FM \cite{blanc2010improving,blanc2023homogenization},
    the associated corrector problem  is formulated as follows
    \begin{equation}\label{eqn:filter_problem}\left\{\begin{aligned}
&-\operatorname{div} \left[\varphi_L(\bm{x}) \left(A_{\mathrm{q-per}}(\bm{x})(\bm{p} + \nabla u_{\bm{p}}) + \bm{\lambda}\right)\right] = 0, \quad \bm{x} = (x_1, x_2) \in Q_L, \\
&\int_{Q_L} \nabla u_{\bm{p}} (\bm{x}) \, \varphi_L(\bm{x}) \, d\bm{x} = 0,
\end{aligned}
\right.
\end{equation}
where $Q_L = LQ$ with $Q$ being the unit cube. The vector $\bm{\lambda}\in\mathbb{R}^2$ 
is a Lagrange multiplier associated with the constraint in the second line of \eqref{eqn:filter_problem}. The filtering function $\varphi_L$ satisfies
  \begin{equation} 
\left\{\begin{aligned}
&\varphi_L \in C^s(\overline{Q}_L), \quad \varphi_L \geq 0, \quad \int_{Q_L} \varphi_L = 1, \\
&\forall l \leq s-1, \quad D^l \varphi_L\big|_{\partial Q_L} = 0.
\end{aligned}
\right.
\end{equation}
The homogenized coefficients obtained by the FM are defined by
  \begin{equation}
    A^*_{ij}=\int_{Q_L}\bm{e}_i^{T}A_{\mathrm{q-per}}(\bm{x})(\bm{e}_j+\nabla u_{\bm{e}_j}(\bm{x}))\varphi_L(\bm{x})\,d\bm{x}.
\end{equation} }
  
 {\Cref{fig:hom_three} illustrates the quasiperiodic homogenized coefficients computed using the three methods: PM, PAM and FM.  The red line represents $A^*_{11, \mathrm{ref}}$ calculated by PM, serving as the reference solution. The blue and black lines show the approximations $A^*_{11, L}$ obtained by PAM and FM, respectively, as the domain size $L$ varies. The corresponding numerical errors $e^*_{11, L}=|A^*_{11, L}-A^*_{11, \mathrm{ref}}|$ are displayed in \Cref{fig:hom_compare_three}. } 
  
  {From these numerical tests, we notice a slow convergence of the PAM in computing the quasiperiodic homogenized coefficient. Furthermore,   \Cref{fig:hom_compare_three} shows that the numerical error of the PAM does not consistently decrease as $L$ increases, mainly due to the unavoidable impact of the Diophantine approximation error. Meanwhile, although FM suppresses oscillations, as shown in \Cref{fig:hom_three},  its numerical error still fails to decrease consistently, as seen in \Cref{fig:hom_compare_three}, due to the same underlying cause. In summary, these results demonstrate that neither the PAM nor the FM can eliminate the impact of Diophantine errors, whereas the PM is capable of fundamentally overcoming this limitation.
    \begin{figure}[!hbpt]	
       \hspace{12mm}
	\subfigure[] 
	{
		\begin{minipage}{7cm}
			\centering         
			\includegraphics[width=6.8cm]{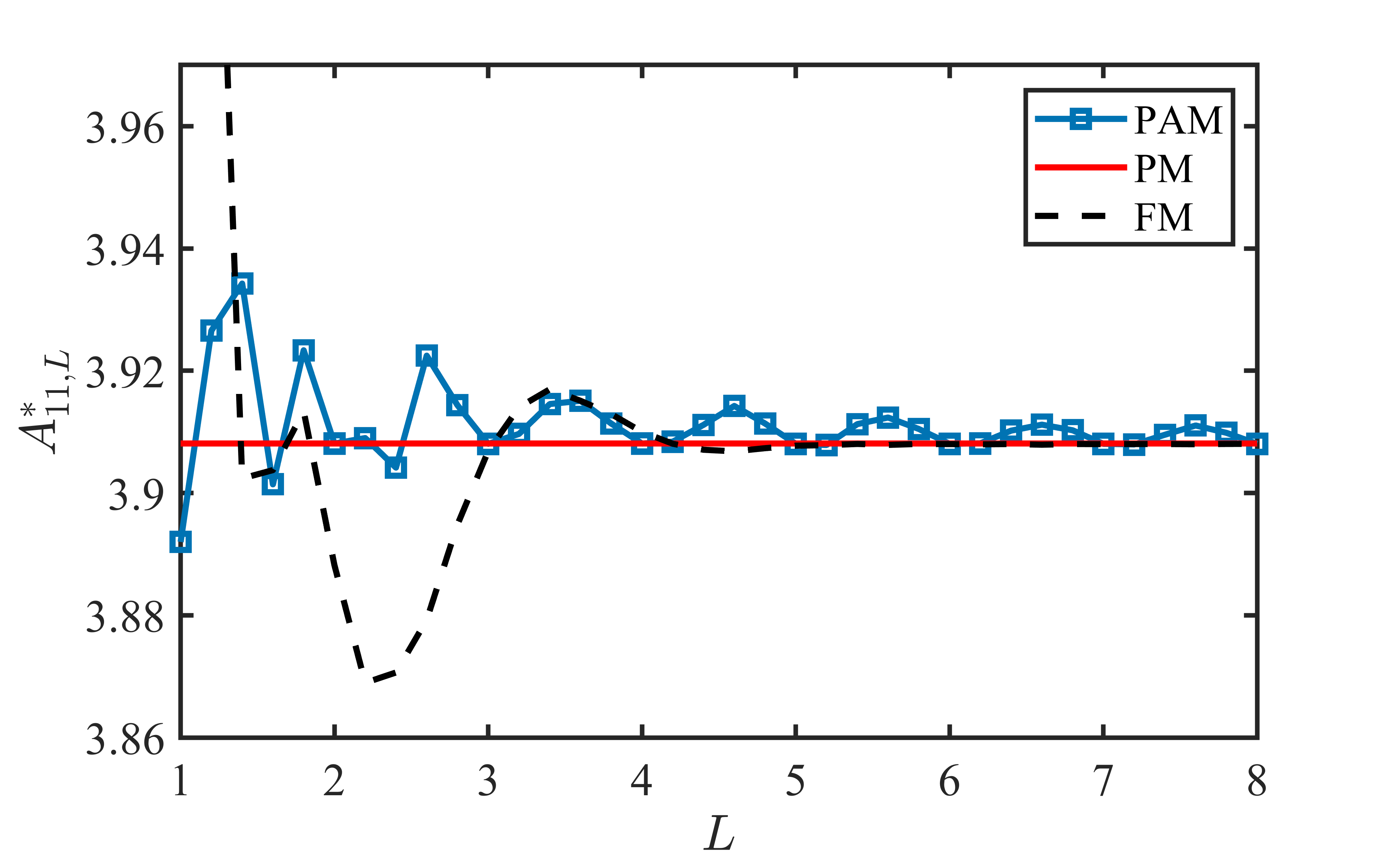}   
            \label{fig:hom_three}
		\end{minipage}
 	}
        \hspace{-9mm}
	\subfigure[]
	{
		\begin{minipage}{7cm}
			\centering      
			\includegraphics[width=6.8cm]{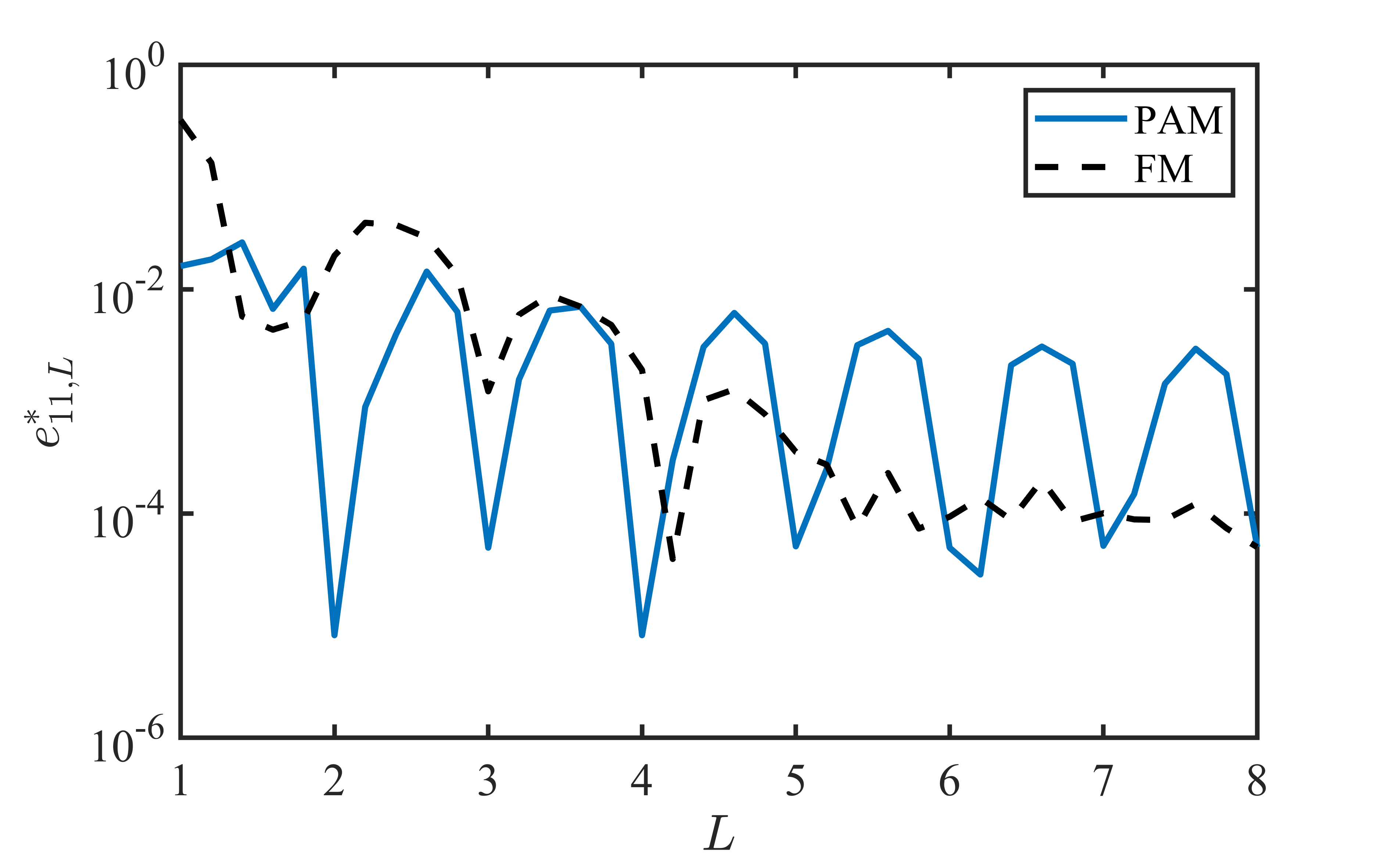}   
            \label{fig:hom_compare_three}
		\end{minipage}
	}
    \caption{{(a) Approximation values of $A^*_{11}$ calculated by PM, PAM and FM; (b) Numerical error $e_{11, L}^*$ of PAM and FM in calculating  $A^*_{11, L}$ against $L$.} }
    \label{fig:hom:filter}
\end{figure}
}

{In the following, we compare the efficiency of PM, PAM and FM for quasiperiodic homogenization. \Cref{tab:hom:four_modes:PMPAMa11compare} showcases the error $e^*_{11, L}$ computed by the three methods, along with the corresponding CPU time. We denote by $N_c$ the number of discretization points per dimension within a unit cell. Consequently, the degrees of freedom (DOF) for PM, PAM and FM are $N_c^4$, $(L N_c)^2$, and $(L N_c)^2$, respectively. As shown in \Cref{tab:hom:four_modes:PMPAMa11compare}, PM demonstrates superior computational efficiency and numerical precision compared to PAM and FM. Specifically, at $N_c = 16$, the PM achieves a high precision ($3.18\times10^{-13}$) in merely 2.89 seconds. In contrast, the PAM and FM plateau at error levels of approximately $3 \times 10^{-5}$ and $1 \times 10^{-4}$, respectively, while requiring significantly longer computation times (up to 10.83 seconds for the FM with $L = 7$). Furthermore, PM exhibits a significantly fast convergence speed as $N_c$ increases, whereas PAM and FM errors do not uniformly decrease as $N_c$ increases. These results demonstrate the high efficiency of the PM.}

       

		


\begin{table}[!hbpt]
	\centering
  \footnotesize{
	\begin{tabular}{|c|c|c|c|c|}
		\hline
		&$N_c$&4&8&16\\
		\hline
		\multirow{5}{*}{$e^*_{11, L}$} & \mbox{PM} & 1.51e-03 & 7.07e-07 & 3.18e-13\\
		\cline{2-5}
		&\mbox{PAM($L=5$)}&1.51e-04 &3.58e-05 &3.55e-05\\
		&\mbox{PAM($L=7$)}& 1.12e-04 &3.47e-05 & 3.42e-05\\
       
        \cline{2-5}
		&\mbox{FM($L=5$)}&1.21e-03 &3.54e-04 &3.50e-04\\
		&\mbox{FM($L=7$)}& 9.71e-05 &1.01e-04 & 1.00e-04\\

		\hline
		\multirow{5}{*}{\mbox{CPU time}}&\mbox{PM} & 1.63e-02& 5.64e-01 & 2.89\\
		\cline{2-5} 
		& \mbox{PAM($L=5$)} &2.97e-03&3.23e-02 &1.94e-01 \\
		& \mbox{PAM($L=7$)} &5.27e-03 &6.84e-02& 5.83e-01\\
		
        \cline{2-5}
		&\mbox{FM($L=5$)}&2.00e-02 &1.01e-01 &1.72\\
		&\mbox{FM($L=7$)}& 3.8e-02 &4.23e-01 & 10.83\\
\hline

	\end{tabular}
 }
 \caption{{Efficiency comparison for computing the homogenized coefficient $A_{11,L}^*$ using different methods: PM (DOF $= N_c^4$), PAM (DOF $= (LN_c)^2$, with $L=5, 7$), and FM (DOF $= (LN_c)^2$, with $L=5,7$).}}
\label{tab:hom:four_modes:PMPAMa11compare}
\end{table}

\section{Conclusion and outlook}\label{sec:conclusion}
This paper focuses on the development of an efficient algorithm for solving quasiperiodic elliptic equations, and in particular, its application in quasiperiodic homogenization. We discretize the equation using the PM, leading to a linear system through tensor-vector-index conversion, and provide a rigorous convergence analysis showing {polynomial accuracy (which, for analytic solutions, can achieve spectral accuracy)}. To reduce computational costs, we propose a compressed storage scheme and use the C-PCG iteration with a diagonal preconditioner. Memory analysis and complexity analysis are provided to assess the algorithm's efficiency. We apply our algorithm to solve quasiperiodic elliptic equations and compute the homogenized coefficients of a multiscale quasiperiodic PDE. Numerical results demonstrate the effectiveness of our algorithm, showing it successfully avoids the Diophantine approximation error and further validates the theoretical results.

In the future, we plan to extend our method to other PDEs with quasiperiodic coefficients, such as parabolic and wave equations. Moreover, we will conduct a detailed convergence analysis of the PM for quasiperiodic PDEs to deepen understanding of its theoretical properties. Finally, we aim to adapt the PM for nonlinear homogenization problems, addressing new challenges and computational strategies. All these projects will not only enhance the theoretical foundation of our method but also contribute to its practical application in quasiperiodic PDEs.


 

\appendix
\section{The proof of \Cref{lemma_appendix:eqv_mod_thm}}
\label{sec_appendix:eqv_mod_thm}
\begin{proof}
We start from the left inequality in \eqref{eqn:eq_mod_thm}. Since $L_i, i=1,2,\cdots, K$ are bounded linear functionals in $\mathcal{H}_{QP}^s(\mathbb{R}^d)$, then we have
$$
    |L_i(u)|\leq \alpha_i\|u\|_s,\quad i=1,2,\cdots, K.
$$
Combining with 
$$
    |u|_s\leq\|u\|_s,
$$
we have 
$$
|u|_s+\sum_{i=1}^K|L_i(u)|\leq\bigg(1+\sum_{i=1}^K\alpha_i\bigg)\|u\|_s.
$$
The conclusion holds by setting $c_1 = \left(1+\sum_{i=1}^K\alpha_i\right)^{-1}$.

To establish the right inequality in \eqref{eqn:eq_mod_thm}, we utilize a proof by contradiction. If it is not true, then for any $n\in\mathbb{N}^+$, there must exist $v_n\in \mathcal{H}_{QP}^s(\mathbb{R}^d)$ such that
$$    \|v_n\|_s>n\bigg(|v_n|_s+\sum_{i=1}^K|L_i(v_n)|\bigg).
$$
    
Let $u_n=v_n/\|v_n\|_s$, then $u_n$ satisfies
\begin{equation}\label{eqn_appendix:eq_mod_thm_1}
        \|u_n\|_s =1,
\end{equation}
and
\begin{equation}\label{eqn_appendix:eq_mod_thm_2}
       |u_n|_s+\sum_{i=1}^K|L_i(u_n)|<\frac{1}{n}.
\end{equation}
    
Firstly, $\left\{u_n\right\}$ is a bounded sequence in $\mathcal{H}_{QP}^s(\mathbb{R}^d)$.  Therefore, there must be a subsequence (still denoted as $u_n$) that converges in $\mathcal{H}_{QP}^{0}(\mathbb{R}^d)$, then we have
$$
    \|u_n-u_m\|_{0}\to 0,\quad n, m\to\infty. 
$$
Using \eqref{eqn_appendix:eq_mod_thm_2}, it holds that
$$
    |u_n-u_m|_s\leq|u_n|_s + |u_m|_s<\frac{1}{n}+\frac{1}{m}\to 0,\quad n,m\to\infty,
$$
then
$$
    \|u_n-u_m\|_s=\left(\|u_n-u_m\|_{0}^2+|u_n-u_m|_s^2\right)^{1/2}\to 0, \quad n,m\to\infty.
$$
Consequently, $\left\{u_n\right\}$ is a Cauchy sequence in $\mathcal{H}_{QP}^s(\mathbb{R}^d)$ and there exists a $u\in \mathcal{H}_{QP}^s(\mathbb{R}^d)$ such that
$$
    \|u_n-u\|_s\to 0,\quad |u_n-u|_s\to 0,\quad n\to\infty.
$$
Taking the limit on both sides of \eqref{eqn_appendix:eq_mod_thm_1}, we have
$$
    \|u\|_s=1,
$$
    i.e., $u\neq 0$.  Taking the limit on both sides of \eqref{eqn_appendix:eq_mod_thm_2}, we have
$$
     |u|_s=0,\quad \sum_{i=1}^{K}|L_i(u)|=0.
$$
That is,  for all $|p|=s$, we have $D^p u=0$. Then we conclude that $u$ is a polynomial of degree $s-1$ such that $|L_i(u)|=0$ for $i=1,2,\cdots, K$. However, this contradicts the assumption. Therefore, we complete the proof.
\end{proof}

\section{The proof of \Cref{lemma_appendix:Poincaré's inequality}}
\label{sec_appendix:poincaré_inequality}
 \begin{proof}
Applying \Cref{lemma_appendix:eqv_mod_thm} with $s=1$ and $L(u)=\displaystyle{\bbint} ud\bm{x}$,  we have
    \begin{equation}
            \begin{aligned}
                |L(u)|&=\left|\bbint ud\bm{x}\right|\leq\bigg(\bbint1^2d\bm{x}\bigg)^{1/2}
                \bigg(\bbint u^2d\bm{x}\bigg)^{1/2}\\
                &=\|u\|_0\leq\|u\|_1
            \end{aligned}
    \end{equation}
    by using the H\"{o}lder's inequality. Therefore,  $L(\cdot)$ is a bounded linear functional on $\mathcal{H}^1_{QP}(\mathbb{R}^d)$. And for any non-zero constant $b$,  we have
        $$
         L(b)=\bbint bd\bm{x}=b\bbint 1 d\bm{x}=b\neq 0.
        $$
        By \Cref{lemma_appendix:eqv_mod_thm}, we have 
        $$
        \|u\|_1\leq c_2\Big(|u|_1 +\bbint ud\bm{x}\Big).
        $$
        Moreover, $u\in\overline{\mathcal{H}}^1_{QP}(\mathbb{R}^d)$ implies $\displaystyle\bbint ud\bm{x}=0$, then   we complete the proof.
    \end{proof}
\section{PAM discretization}
\label{sec_appendix:pam_discrete}

Here, we utilize the PAM to discretize \eqref{eqn:elliptic}. We select functions $\alpha_p(\bm{x})$ and $f_p(\bm{x})$ with a period of $L$ to approximate $\alpha(\bm{x})$ and $f(\bm{x})$,  respectively.
Define the following sets
$$
\Lambda_{p,\alpha}=\left\{\bm{h}_\alpha=[L\bm{\lambda}_\alpha],~~\bm{\lambda}_\alpha\in\Lambda_\alpha\right\},
$$
$$
\Lambda_{p,f}=\left\{\bm{h}_f=[L\bm{\lambda}_f],~~\bm{\lambda}_f\in\Lambda_f\right\},
$$
where  $[\bm{x}]=([x_1],\cdots,[x_d])$. Then, we can express the Fourier series of $\alpha_p(\bm{x})$ and $f_p(\bm{x})$ as follows
$$
\alpha_p(\bm{x})=\sum_{\bm{h}_\alpha\in\Lambda_{p,\alpha}}\hat{\alpha}_p(\bm{h}_\alpha)e^{\imath(\bm{h}^T_\alpha\bm{x})/L},\quad \hat{\alpha}_p(\bm{h}_\alpha)=\frac{1}{|\mathbb{T}^d|}\int_{\mathbb{T}^d}\alpha_p(\bm{x})e^{-\imath(\bm{h}^T_\alpha\bm{x})/L},
$$
$$
f_p(\bm{x})=\sum_{\bm{h}_f\in\Lambda_{p,f}}\hat{f}_p(\bm{h}_f)e^{\imath(\bm{h}^T_f\bm{x})/L}, \quad\hat{f}_p(\bm{h}_f)=\frac{1}{|\mathbb{T}^d|}\int_{\mathbb{T}^d}f_p(\bm{x})e^{-\imath(\bm{h}^T_u\bm{x})/L},
$$
where $\mathbb{T}^d= [0,2\pi L)^d$. As in  the quasiperiodic case, $\Lambda_{p, u}\subseteq\mathrm{span}_{\mathbb{Z}}\left\{\Lambda_{p,\alpha},\Lambda_{p,f}\right\}$, we then obtain 
$$
u_p(\bm{x})=\sum_{\bm{h}_u\in\Lambda_{p,u}}\hat{u}_p(\bm{h}_u)e^{\imath(\bm{h}^T_u\bm{x})/L}, \quad\hat{u}_p(\bm{h}_u)=\frac{1}{|\mathbb{T}^d|}\int_{\mathbb{T}^d}u(\bm{x})e^{-\imath(\bm{h}^T_u\bm{x})/L}.
$$
Next, we discretize $\mathbb{T}^d$ as 
$$
\mathbb{T}^d_N=\left\{\bm{x_j}=2\pi(Lj_1/N, \cdots, L j_d/N)\in \mathbb{T}^d: 
~0\leq j_1,\cdots, j_d\leq N-1,~ \bm{j}=(j_1,\cdots , j_d)\right\}.
$$
This indicates that we distribute $N$ points uniformly in each spatial direction of $\mathbb{T}^d$. As a result, there are a total of $E = N^d$ discrete points.

Similar to the definitions of the grid function space $S_N$ and the inner product $(\cdot , \cdot)_N$ mentioned in \Cref{sec:numerical_methods}, we can obtain the discrete Fourier coefficients of a periodic function $F(\bm{x})$ as follows
$$ 
\widetilde{F}(\bm{h}_F)=\big(F(\bm{x_j}),e^{\imath(\bm{h}^T_F\bm{x_j})/L}\big)_N=\frac{1}{E}\sum_{\bm{h}_F\in \Lambda_{p,F, N}}F(\bm{x_j})e^{-\imath(\bm{h}^T_F\bm{x_j})/L},\quad \bm{x_j}\in \mathbb{T}_N^d,
$$
where $\Lambda_{p,F, N}=\left\{\bm{h}_F :~\bm{h}_F=[L\bm{\lambda}_F],~\bm{\lambda}_F\in K_N^d\right\}$.
Then we can apply PAM to discretize $\alpha_p,u_p,f_p$, respectively,
$$
\alpha_p(\bm{x_j})=\sum_{\bm{h}_\alpha\in\Lambda_{p,\alpha, N}}\widetilde{\alpha}_p(\bm{h}_\alpha)e^{\imath(\bm{h}_\alpha^T\bm{x_j})/L},~~\widetilde{\alpha}_p(\bm{h}_\alpha)=\frac{1}{E}\sum_{\bm{h}_\alpha\in\Lambda_{p,\alpha, N}}\alpha_p(\bm{x_j})e^{-\imath(\bm{h}^T_\alpha\bm{x_j})/L}, \quad \bm{x_j}\in \mathbb{T}_N^d,
$$
$$
u_p(\bm{x_j})=\sum_{\bm{h}_u\in\Lambda_{p,u,N}}\widetilde{u}_p(\bm{h}_u)e^{\imath(\bm{h}_u^T\bm{x_j})/L},~~\widetilde{u}_p(\bm{h}_u)=\frac{1}{E}\sum_{\bm{h}_u\in\Lambda_{p,u,N}}u_p(\bm{x_j})e^{-\imath(\bm{h}^T_u\bm{x_j})/L},\quad \bm{x_j}\in \mathbb{T}_N^d,
$$
$$
f_p(\bm{x_j})=\sum_{\bm{h}_f\in\Lambda_{p,f,N}}\widetilde{f}_p(\bm{h}_f)e^{\imath(\bm{h}_f^T\bm{x_j})/L},~~\widetilde{f}_p(\bm{h}_f)=\frac{1}{E}\sum_{\bm{h}_f\in\Lambda_{p,f,N}}f_p(\bm{x_j})e^{-\imath(\bm{h}^T_f\bm{x_j})/L},\quad \bm{x_j}\in \mathbb{T}_N^d.
$$
To simplify the derivation, we introduce the notation $\widetilde{\bm{h}}_\alpha=\bm{h}_\alpha/L\in\widetilde{\Lambda}_{p,\alpha,N}$. Similarly, $\widetilde{\bm{h}}_u$ and $\widetilde{\bm{h}}_f$ are defined in the same manner.

Given a test function $v_p \in V^N := \mathrm{span}\big\{e^{\imath{\bm{\widetilde{h}}_v}^T\bm{x}}:~\bm{\widetilde{h}}_v\in\widetilde{\Lambda}_{p,u,N},~\bm{x}\in\mathbb{T}^d\big\}$, the discrete variational formulation is to seek $u_p \in V^N$,  such that 
$$
(\alpha_p\nabla u_p,\nabla v_p)_N=(f_p,v_p)_N\quad \forall v_p\in V^N.
$$
From the orthogonality of base functions, we have
\begin{equation}\nonumber
	\begin{aligned}
		(\alpha_p\nabla u_p,\nabla v_p)_N&=\bigg(\sum_{\bm{\widetilde{h}}_\alpha\in\widetilde{\Lambda}_{p,\alpha,N}}\widetilde{\alpha}_p(\bm{\widetilde{h}}_\alpha)e^{\imath\bm{\widetilde{\bm{h}}^T_\alpha}\bm{x_j}}\sum_{\bm{\bm{\widetilde{h}}_u}\in\widetilde{\Lambda}_{p,u,N}}(\imath{\widetilde{\bm{h}}^T_u})\widetilde{u}_p(\widetilde{\bm{h}}_u)e^{\imath\bm{\widetilde{h}}_u^T\bm{x_j}},(\imath\bm{\widetilde{h}}^T_v)e^{\imath\bm{\widetilde{h}}^T_v\bm{x_j}}\bigg)_N\\
		&=\bigg(\sum_{\bm{\widetilde{h}}={\bm{\widetilde{h}}_\alpha+\bm{\widetilde{h}}_u}}\sum_{\bm{\widetilde{h}}_u\in \widetilde{\Lambda}_{p,u,N}}\widetilde{\alpha}_p(\bm{\widetilde{h}}_v-\bm{\widetilde{h}}_u)\widetilde{u}_p(\bm{\widetilde{h}}_u)(\imath\bm{\widetilde{h}}^T_u)e^{\imath\bm{\widetilde{h}}^T\bm{x_j}}, (\imath\bm{\widetilde{h}}^T_v) e^{\imath\bm{\widetilde{h}}_v^T\bm{x_j}}\bigg)_N\\
		&=\sum_{\bm{\widetilde{h}}_u\in \widetilde{\Lambda}_{p,u,N}}\widetilde{\alpha}_p({\bm{\widetilde{h}}_v\overset{N}{-}\bm{\widetilde{h}}_u})\widetilde{u}_p(\bm{\widetilde{h}}_u)(\bm{\widetilde{h}}_v^T\bm{\widetilde{h}}_u),
	\end{aligned}
\end{equation}
\begin{equation}\nonumber
	(f_p,v_p)_N=\bigg(\sum_{\bm{\widetilde{h}}_f\in \widetilde{\Lambda}_{p,f,N}}\widetilde{f}_p(\bm{\widetilde{h}}_f)e^{\imath\bm{\widetilde{h}}_f^T\bm{x_j}},e^{\imath\bm{\widetilde{h} }_v^T\bm{x_j}}\bigg)_N=\widetilde{f}_p(\bm{\widetilde{h}}_v),
\end{equation}
that is
$$
\sum_{\bm{\widetilde{h}}_u\in \widetilde{\Lambda}_{p,u,N}}\widetilde{\alpha}_p(\bm{\widetilde{h}}_v\overset{N}{-}\bm{\widetilde{h}}_u)\widetilde{u}_p(\bm{\widetilde{h}}_u)(\bm{\widetilde{h}}_v^T\bm{\widetilde{h}}_u)=\widetilde{f}_p(\bm{\widetilde{h}}_v).
$$

Similar to the quasiperiodic case, we can employ tensor-vector-index conversion to generate the linear system. Define
$$
A=(A_{ij})\in\mathbb{C}^{E\times E},\quad A_{ij}= \widetilde{\alpha}_p(\bm{\widetilde{h}}_v\overset{N}{-}\bm{\widetilde{h}}_u),
$$
$$
W=(W_{ij})\in\mathbb{C}^{E\times E},\quad W_{ij}=\bm{\widetilde{h}}_v^T\bm{\widetilde{h}}_u,
$$
where indices $i, j$ are determined by
\begin{equation}\label{eqn_appendix:ij}
    \bm{h}_{v}\xrightarrow{\mathcal{C}} i, \quad \bm{h}_{u}\xrightarrow{\mathcal{C}} j,
\end{equation}
respectively. The column vectors $\bm{U}$ and $\bm{F}$ are defined by
$$
\bm{U}=(U_j)\in\mathbb{C}^{E},\quad U_j=\widetilde{u}_p(\bm{\widetilde{h}}_u),
$$
$$
\bm{F}=(F_i)\in\mathbb{C}^{E},\quad F_i=\widetilde{f}_p(\bm{\widetilde{h}}_v).
$$
Finally, we obtain the following linear system for PAM
$$
Q\bm{U}=\bm{F},\quad Q=A\circ W\in\mathbb{C}^{E\times E}.
$$

\section{Solving corrector equation by the PM}
\label{sec_appendix:solve_corrector_eqn}
In this section , we are concerned about the two-dimensional  quasiperiodic corrector equation 
\begin{align}\label{eqn_appendix:corrector}
    -\mathrm{div}(A_{\mathrm{q-per}}(\bm{x})(\bm{p}+\nabla u_{\bm{p}}(\bm{x}))) = 0,
\end{align}
where $\bm{p}=\bm{e}_i, i=1,2$. We apply PM to obtain $u_{\bm{e}_i}$ and calculate the homogenized coefficient  $A^*$ using the formula 
\begin{equation}\label{eqn_appendix:homcoefficient}
    A^*_{ij}=\mathcal{M}\left\{\bm{e}_i^{T}A_{\mathrm{q-per}}(\bm{x})(\bm{e}_j+\nabla u_{\bm{e}_j}(\bm{x}))\right\}.
\end{equation}

We assume that 
$$
 A_{\mathrm{q-per}}(\bm{x})=\left(
 \begin{array}{cc}
 	\alpha(\bm{x}) & 0\\
 	0& \beta(\bm{x})
 \end{array}
 \right),
$$
where $\alpha(\bm{x}), \beta(\bm{x})$ have the same projection matrix $\bm{P} = (\bm{P}_1, \bm{P}_2)^T \in\mathbb{P}^{2\times n}$. \eqref{eqn_appendix:corrector} can be split into two separate equations 
\begin{align}\label{eqn_appendix:elliptic1}
	 -\mathrm{div}(A_{{\mathrm{q-per}}}(\bm{x})\nabla u_{\bm{e}_1}(\bm{x}))= f_1(\bm{x}),
\end{align}

\begin{align}\label{eqn_appendix:elliptic2}
	-\mathrm{div}(A_{{\mathrm{q-per}}}(\bm{x})\nabla u_{\bm{e}_2}(\bm{x}))= f_2(\bm{x}),
\end{align}
where 
$$
 f_1(\bm{x}) = \partial_{x_1}\alpha(\bm{x}), 
\quad f_2(\bm{x}) = \partial_{x_2}\beta(\bm{x}).
 $$
We only derive the discrete scheme for  \eqref{eqn_appendix:elliptic1}, while the scheme for  \eqref{eqn_appendix:elliptic2} follows a similar approach.

Applying PM to discretize \eqref{eqn_appendix:elliptic1}, we have
$$
\alpha(\bm{x_j})=\sum_{\bm{{k}_\alpha}\in K^n_N}\widetilde{\alpha}(\bm{k}_\alpha)e^{\imath(\bm{Pk}_\alpha)^T\bm{x_j}}\quad \bm{x_j}\in X_{\bm{P}},	
$$	
$$
\beta(\bm{x_j})=\sum_{\bm{\bm{k}_\beta}\in K^n_N}\widetilde{\alpha}(\bm{k}_\beta)e^{\imath(\bm{Pk}_\beta)^T\bm{x_j}},\quad \bm{x_j}\in X_{\bm{P}},	
$$	
$$
u_{\bm{e}_1}(\bm{x_j})=\sum_{\bm{\bm{k}_{u_{\bm{e}_1}}}\in K^n_N}\widetilde{u}_{\bm{e}_1}(\bm{k}_{u_{\bm{e}_1}})e^{\imath(\bm{Pk}_{u_{\bm{e}_1}})^T\bm{x_j}},\quad \bm{x_j}\in X_{\bm{P}},
$$
$$
f_1(\bm{x_j})=\sum_{{\bm{k}_{f_1}}\in K^n_N}\widetilde{f_1}(\bm{k}_{f_1})e^{\imath(\bm{Pk}_{f_1})^T\bm{x_j}}, \quad \bm{x_j}\in X_{\bm{P}}.
$$

Given an arbitrary test function $v\in \overline{V}^N$, the discrete variational formulation is to seek  $u_{\bm{e}_1}\in \overline{V}^N$  such that
$$
(A_{{\mathrm{q-per}}}\nabla u_{\bm{e}_1},\nabla v)_N = (f_1 ,v)_N.
$$
In consideration of the orthogonality of base functions, we have
\begin{align*}
	& (A_{{\mathrm{q-per}}}\nabla u_{\bm{e}_{1}},\nabla v)_N
   \\
   &=
	\left(
	\begin{pmatrix}
		\sum\limits_{\bm{k}_{\alpha}\in K^n_N}\widetilde{\alpha}(\bm{k}_{\alpha})e^{\imath(\bm{Pk}_{\alpha})^{T}\bm{x_j}}&0\\
		0 & \sum\limits_{\bm{k}_{\beta}\in K^n_N}\widetilde{\beta}(\bm{k}_{\beta})e^{\imath(\bm{Pk}_{\beta})^{{T}}\bm{x_j}}
	\end{pmatrix}
	\begin{pmatrix}
		\sum\limits_{\bm{k}_{u_{\bm{e}_1}}\in K^n_N}\widetilde{u}_{\bm{e}_{1}}(\bm{k}_{u_{\bm{e}_1}})(\imath\bm{P}_{1}\bm{k}_{u_{\bm{e}_1}})e^{\imath(\bm{Pk}_{u_{\bm{e}_1}})^{{T}}\bm{x_j}}\\
		\sum\limits_{\bm{k}_{u_{\bm{e}_1}}\in K^n_N}\widetilde{u}_{\bm{e}_{1}}(\bm{k}_{u_{\bm{e}_1}})(\imath\bm{P}_{2}\bm{k}_{u_{\bm{e}_1}})e^{\imath(\bm{Pk}_{u_{\bm{e}_1}})^{{T}}\bm{x_j}}
	\end{pmatrix}
	\right.,\\
	&\left.
	\begin{pmatrix}
		(\imath\bm{P}_{1}\bm{v})e^{\imath(\bm{Pk}_v)^{{T}}\bm{x_j}}\\
		(\imath\bm{P}_{2}\bm{v})e^{\imath(\bm{Pk}_v)^{{T}}\bm{x_j}}
	\end{pmatrix}
	\right)_{N}\\
	=&\left(
	\begin{pmatrix}
		\sum\limits_{\bm{k}_{\alpha}\in K^n_N}\widetilde{\alpha}(\bm{k}_{\alpha})e^{\imath(\bm{Pk}_{\alpha})^{{T}}\bm{x_j}}\sum\limits_{\bm{k}_{u_{\bm{e}_1}}\in K^n_N}\widetilde{u}_{\bm{e}_{1}}(\bm{k}_{u_{\bm{e}_1}})(\imath\bm{P}_{1}\bm{k}_{u_{\bm{e}_1}})e^{\imath(\bm{P}\bm{k}_{u_{\bm{e}_1}})^{{T}}\bm{x_j}}\\
		\sum\limits_{\bm{k}_{\beta}\in K^n_N}\widetilde{\beta}(\bm{k}_{\beta})e^{\imath(\bm{Pk}_{\beta})^{{T}}\bm{x_j}}\sum\limits_{\bm{k}_{u_{\bm{e}_1}}\in K^n_N}\widetilde{u}_{\bm{e}_{1}}(\bm{k}_{u_{\bm{e}_1}})(\imath\bm{P}_{2}\bm{k}_{u_{\bm{e}_1}})e^{\imath(\bm{P}\bm{k}_{u_{\bm{e}_1}})^{{T}}\bm{x_j}}
	\end{pmatrix}
    ,
	\begin{pmatrix}
		(\imath\bm{P}_{1}\bm{k}_v)e^{\imath(\bm{Pk}_v)^{{T}}\bm{x_j}} \\
		(\imath\bm{P}_{2}\bm{k}_v)e^{\imath(\bm{Pk}_v)^{{T}}\bm{x_j}}
	\end{pmatrix}
	\right)_{N}\\
	&=\sum\limits_{\bm{k}_{u_{\bm{e}_1}}\in K^n_N}\widetilde{\alpha}(\bm{k}_v\overset{N}{-}\bm{k}_{u_{\bm{e}_1}})\widetilde{u}_{\bm{e}_{1}}(\bm{k}_{u_{\bm{e}_1}})(\bm{P}_{1}\bm{k}_v)^{{T}}(\bm{P}_{1}\bm{k}_{u_{\bm{e}_1}})+\sum\limits_{\bm{k}_{u_{\bm{e}_1}}\in K^n_N}\widetilde{\beta}(\bm{k}_v\overset{N}{-}\bm{k}_{u_{\bm{e}_1}})\widetilde{u}_{\bm{e}_{1}}(\bm{k}_{u_{\bm{e}_1}})(\bm{P}_{2}\bm{k}_v)^{{T}}(\bm{P}_{2}\bm{k}_{u_{\bm{e}_1}}),
\end{align*}

\begin{equation}\nonumber
	(f_1,v)_N=\bigg(\sum_{\bm{k}_{f_1}\in K^n_N}\widetilde{f}(\bm{k}_{f_1})e^{\imath(\bm{Pk }_{f_1})^T\bm{x_j}},e^{\imath(\bm{Pk }_v)^T\bm{x_j}}\bigg)_N=\widetilde{f_1}(\bm{k}_v),
\end{equation} 
that is
\begin{align*}
    \sum\limits_{\bm{k}_{u_{\bm{e}_1}}\in K^n_N}\widetilde{\alpha}(\bm{k}_v\overset{N}{-}\bm{k}_{u_{\bm{e}_1}})&\widetilde{u}_{\bm{e}_{1}}(\bm{k}_{u_{\bm{e}_1}})(\bm{P}_{1}\bm{k}_v)^{{T}}(\bm{P}_{1}\bm{k}_{u_{\bm{e}_1}})
	+\\
    &\sum\limits_{\bm{k}_{u_{\bm{e}_1}}\in K^n_N}\widetilde{\beta}(\bm{k}_v\overset{N}{-}\bm{k}_{u_{\bm{e}_1}})\widetilde{u}_{\bm{e}_{1}}(\bm{k}_{u_{\bm{e}_1}})(\bm{P}_{2}\bm{k}_v)^{{T}}(\bm{P}_{2}\bm{k}_{u_{\bm{e}_1}})
	=\widetilde{f_1}(\bm{k}_v).
\end{align*}
Define
$$
A^1=(A^1_{ij})\in\mathbb{C}^{D\times D},\quad A^1_{ij}= \widetilde{\alpha}_{\bm{k}_v\overset{N}{-}\bm{k}_{u_{\bm{e}_1}}},
$$
$$
A^2=(A^2_{ij})\in\mathbb{C}^{D\times D},\quad A^2_{ij}= \widetilde{\beta}_{\bm{k}_v\overset{N}{-}\bm{k}_{u_{\bm{e}_1}}},
$$
$$
W^1=(W^1_{ij})\in\mathbb{C}^{D\times D},\quad W^1_{ij}= (\bm{P}_1\bm{{k}}_v)^T(\bm{P}_1\bm{{k}}_{u_{\bm{e}_1}}),
$$
$$
W^2=(W^2_{ij})\in\mathbb{C}^{D\times D},\quad W^2_{ij}= (\bm{P}_2\bm{{k}}_v)^T(\bm{P}_2\bm{{k}}_{u_{\bm{e}_1}}),
$$
where indices $i, j$ are determined by
\begin{equation}
    \bm{k}_{v}\xrightarrow{\mathcal{C}}i, \quad \bm{k}_{u_{\bm{e}_1}}\xrightarrow{\mathcal{C}} j,
\end{equation}
respectively. The column vectors $\bm{U}$ and $\bm{F}$ are defined, respectively, by
$$
\bm{U}=(U_j)\in\mathbb{C}^{D},\quad U_j=\widetilde{u} _{\bm{e}_1}(\bm{k}_{u_{\bm{e}_1}}),
$$
$$
\bm{F}=(F_i)\in\mathbb{C}^{D},\quad F_i=\widetilde{f}_1(\bm{k}_v).
$$
Consequently, we obtain the following  linear system 
$$
Q\bm{U}=\bm{F},\quad Q=A^1\circ W^1+A^2\circ W^2\in\mathbb{C}^{D\times D}.
$$

To derive the discrete scheme of  \eqref{eqn_appendix:elliptic2}, we replace $f_1$ to $f_2$ and obtain the linear system in the same way. The homogenized coefficients are then obtained by clarifying the formula \eqref{eqn:homcoefficient} 
\begin{align}\label{eqn_appendix:homcoefficient11}
    A^*_{11}=\mathcal{M}\left\{\alpha(\bm{x})\nabla_{x_1} u_{\bm{e}_1}(\bm{x})+\alpha(\bm{x})\right\},
\end{align}
\begin{align}\label{eqn_appendix:homcoefficient12}
    A^*_{12}=\mathcal{M}\left\{\alpha(\bm{x})\nabla_{x_1} u_{\bm{e}_2}(\bm{x})\right\},
\end{align}
\begin{align}\label{eqn_appendix:homcoefficient21}
    A^*_{21}=\mathcal{M}\left\{\beta(\bm{x})\nabla_{x_2} u_{\bm{e}_1}(\bm{x}))\right\},
\end{align}
\begin{align}\label{eqn_appendix:homcoefficient22}
    A^*_{22}=\mathcal{M}\left\{\beta(\bm{x})\nabla_{x_2} u_{\bm{e}_2}(\bm{x})+\beta(\bm{x})\right\}.
\end{align}

\section{Solving corrector equation by the FM}
\label{sec_appendix:solve_corrector_eqn_filter}

In this section, we consider the filtered quasiperiodic corrector problem:
\begin{equation}\label{eqn_appendix:filter_problem}
\left\{
\begin{aligned}
&-\nabla \cdot \left[ \varphi_L(\bm{x}) \left( A_{\mathrm{q\text{-}per}}(\bm{x})(\bm{p} + \nabla u_{\bm{p}}) + \bm{\lambda} \right) \right] = 0, && \bm{x} \in Q_L, \\
&\int_{Q_L} \varphi_L(\bm{x}) \nabla u_{\bm{p}}(\bm{x}) \, d\bm{x} = 0,
\end{aligned}
\right.
\end{equation}
where \( Q_L = LQ \), with $Q$  denoting the unit cell, and \( \bm{p} = \bm{e}_i \in \mathbb{R}^2 \), \( i = 1, 2 \). The vector \( \bm{\lambda} \in \mathbb{R}^2 \) is a Lagrange multiplier enforcing the constraint in the second line of \eqref{eqn_appendix:filter_problem}. The quasiperiodic coefficient matrix is given by
\begin{equation}
A_{\mathrm{q\text{-}per}}(\bm{x}) = \begin{pmatrix}
\alpha(\bm{x}) & 0 \\
0 & \beta(\bm{x})
\end{pmatrix}, \quad \bm{x} \in \mathbb{R}^2.
\end{equation}

To solve the filtered problem \eqref{eqn_appendix:filter_problem} and approximate the homogenized coefficients associated with \( A_{\mathrm{q\text{-}per}} \), we apply the filtering method introduced in~\cite{blanc2010improving}. The filtering function \( \varphi_L \colon Q_L \rightarrow \mathbb{R} \) is a scaled, compactly supported \( C^2 \) function defined as
\begin{equation}
\varphi_L(\bm{x}) = \frac{1}{L^2} \varphi\left( \frac{\bm{x}}{L} \right), \quad \varphi(\bm{x}) = \prod_{i=1}^2 \varphi_0(x_i),
\end{equation}
where  \( \varphi_0 \colon \mathbb{R} \rightarrow \mathbb{R} \) is given by
\begin{equation}
\varphi_0(t) =
\begin{cases}
t^2, & 0 < t < \frac{1}{3}, \\
- \frac{1}{3} + 2t - 2t^2, & \frac{1}{3} \leq t < \frac{2}{3}, \\
(t - 1)^2, & \frac{2}{3} \leq t < 1, \\
0, & \text{otherwise}.
\end{cases}
\end{equation}

The weak formulation of~\eqref{eqn_appendix:filter_problem} reads: find \( (u_{\bm{p}}, \bm{\lambda}) \in (H_L^1/\mathbb{R}) \times \mathbb{R}^2 \) such that
\begin{equation}
\begin{aligned}
&\int_{Q_L} \varphi_L(\bm{x}) \nabla v(\bm{x})^T A_{\mathrm{q\text{-}per}}(\bm{x}) \nabla u_{\bm{p}}(\bm{x}) \, d\bm{x} 
+ \int_{Q_L} \varphi_L(\bm{x}) \bm{\lambda}^T \nabla v(\bm{x}) \, d\bm{x} \\
&\qquad - \int_{Q_L} \varphi_L(\bm{x}) \bm{\mu}^T \nabla u_{\bm{p}}(\bm{x}) \, d\bm{x}
= - \int_{Q_L} \varphi_L(\bm{x}) \nabla v(\bm{x})^T A_{\mathrm{q\text{-}per}}(\bm{x}) \bm{p} \, d\bm{x},
\end{aligned}
\end{equation}
for all \( (v, \bm{\mu}) \in (H_L^1/\mathbb{R}) \times \mathbb{R}^2 \). We refer to \cite[Definition~2.6]{blanc2010improving} for the definition of the space \( H_L^1/\mathbb{R} \). Boundary integrals vanish due to $\varphi_L|_{\partial Q_L} = 0$.

We discretize the weak problem using linear (P1) finite elements. The domain \( Q_L \) is partitioned into a uniform triangular mesh with mesh size \( h = L/N \), resulting in \( N^2 \) nodes. Let \( \{ \phi_i \}_{i=1}^{N^2} \) denote the standard nodal basis functions. The discrete stiffness matrix \( K \in \mathbb{R}^{N^2 \times N^2} \), constraint matrix \( G \in \mathbb{R}^{N^2 \times 2} \), and load vector \( \bm{F} \in \mathbb{R}^{N^2} \) are defined by:
\begin{align*}
K_{ij} &= \int_{Q_L} \varphi_L(\bm{x}) \, \nabla \phi_i(\bm{x})^T A_{\mathrm{q\text{-}per}}(\bm{x}) \, \nabla \phi_j(\bm{x}) \, d\bm{x}, \quad i,j = 1, \dots, N^2, \\
G_{jk} &= \int_{Q_L} \varphi_L(\bm{x}) \, \frac{\partial \phi_j}{\partial x_k}(\bm{x}) \, d\bm{x}, \quad j = 1, \dots, N^2,\quad k = 1,2, \\
F_i &= -\int_{Q_L} \varphi_L(\bm{x}) \, \nabla \phi_i(\bm{x})^T A_{\mathrm{q\text{-}per}}(\bm{x}) \bm{p} \, d\bm{x}, \quad i = 1, \dots, N^2.
\end{align*}
The linear system then reads:
\begin{equation}
\begin{pmatrix}
K & G \\
G^T & 0
\end{pmatrix}
\begin{pmatrix}
\bm{u}_{\bm{p},h} \\
\bm{\lambda}
\end{pmatrix}
=
\begin{pmatrix}
\bm{F} \\
\bm{0}
\end{pmatrix},
\end{equation}
where \( \bm{u}_{\bm{p},h} \in \mathbb{R}^{N^2} \) represents the vector of finite element coefficients approximating \( u_{\bm{p}} \). The filtered approximation of the homogenized coefficient \( A^* \) is then computed by
\begin{equation}
A^*_{ij} = \int_{Q_L} \bm{e}_i^T A_{\mathrm{q\text{-}per}}(\bm{x}) \left( \bm{e}_j + \nabla u_{\bm{e}_j}(\bm{x}) \right) \varphi_L(\bm{x}) \, d\bm{x}, \quad i,j = 1,2.
\end{equation}

\end{document}